\documentclass[a4paper,11pt]{amsart}
\usepackage[arrow,matrix]{xy}
\usepackage{amsmath,amssymb,amscd,bbm,amsthm,mathrsfs,hyperref}
\theoremstyle{plain}
\newtheorem{thm}{Theorem}[section]
\newtheorem{cor}[thm]{Corollary}
\newtheorem{lem}[thm]{Lemma}
\newtheorem{prop}[thm]{Proposition}
\newtheorem{defn}[thm]{Definition}
\newtheorem{exa}[thm]{Example}

\newtheorem{rem}[thm]{Remark}

 \def\s{\sum}   
\def\Hom{\operatorname {Hom}}  
\def\RHom{\operatorname {RHom}}
 
 \def\k{\mathbbm{k}}

\begin{document}
\title{\bf derived Picard groups of homologically smooth Koszul DG algebras}
\author{X.-F. Mao}
\address{Department of Mathematics, Shanghai University, Shanghai 200444, China}
\email{xuefengmao@shu.edu.cn}

\author{Y.-N. Yang}
\address{Department of Mathematics, Shanghai University, Shanghai 200444, China}
\email{mooly@shu.edu.cn}

\author{J.-W. He}
\address{Department of Mathematics, Hangzhou Normal University,  16 Xuelin Rd, Hangzhou Zhejiang 310036, China}
\email{jwhe@hznu.edu.cn}
\date{}

\begin{abstract}
In this paper, we show that the derived Picard group of a homologically smooth Koszul connected cochain DG algebra is isomorphic to the opposite group of the derived Picard group of its finite dimensional local Ext-algebra. As applications, we compute the derived Picard groups of some important DG algebras such as
trivial DG polynomial algebras, trivial DG free algebras and several non-trivial DG down-up algebras and DG free algebras.
\end{abstract}
\maketitle

\section*{Introduction}
The derived Picard group of an algebra is introduced independently by Yekutieli and Rouquier-Zimmermann \cite{Yek1,RZ,Zim}.
Due to the well known Rickard's theory \cite{Ric1,Ric2}, one sees that the derived Picard group of an algebra is an invariant of the derived category of algebras, which are projective over a base ring $R$. In \cite{Kel4}, Keller interprets Hochschild cohomology as the Lie algebra of the derived Picard group and
deduce that it is preserved under derived equivalences.
There has been many works on properties and computations of derived Picard groups of various kind of algebras. For examples, the derived Picard group of a commutative unital ring has been computed in some cases \cite{Har,RZ,Yek1,Fau}, some properties of the derived Picard groups of an order are given in \cite{Zim}, the derived Picard group of a finite dimensional algebra over an algebraic closed field is proved to be a locally algebraic group in \cite{Yek2},  and some structures and calculations of the derived Picard groups of finite dimensional hereditary algebras over an algebraic closed field are presented in \cite{MY}. The case of a few particular algebras were studied by geometric means, using that the derived category of coherent sheaves of certain varieties is equivalent to the derived category of a specific algebra. Recently, Volkov and Zvonareva \cite{VZ} completely determine the derived Picard group of a self-injective Nakayama algebra, which involves braid groups and extension of these.
  It is also interesting to consider the case of DG algebras.
For graded commutative DG algebras, Yekutieli \cite{Yek3} has done some research on their derived Picard groups and dualizing DG modules.

This paper deals with connected cochain DG algebras with particular emphasis on the non-commutative aspects. In the context, there has been many papers on various homological properties of them. Among those properties,  the homologically smoothness of a DG algebra is especially important.  Recall that a connected cochain DG algebra $A$ is called homologically smooth, if ${}_{A}k$, or equivalently the DG $A^e$-module $A$ is compact (cf. \cite[Corollary 2.7]{MW3}).
In the DG context, it is analogous to the regular property of graded algebras.
The first author and Wu \cite{MW2} show that any homologically
smooth connected cochain DG algebra $A$ is cohomologically unbounded
unless $A$ is quasi-isomorphic to the simple algebra $k$.
And it is
proved in \cite{MW2} that the $\mathrm{Ext}$-algebra
 of a homologically smooth DG algebra $A$
is Frobenius if and only if both $\mathscr{D}^b_{lf}(A)$ and
$\mathscr{D}^b_{lf}(A\!^{op})$ admit Auslander-Reiten triangles, where $\mathscr{D}^b_{lf}(A)$ is the full subcategory of $\mathscr{D}(A)$ consisting of DG $A$-modules whose cohomology modules are finite-dimensional (bounded and locally finite-dimensional) as $k$-vector spaces.  In \cite{HW},  the third author and Wu introduce the concept of
Koszul DG algebras. By the definition, a connected cochain DG algebra is called Kozul, if ${}_Ak$, or equivalently ${}_{A^e}A$, has a minimal semi-free resolution with a semi-basis concentrated in degree $0$ (cf. \cite{HW}). For homologically smooth, Gorenstein and Koszul DG algebras, they obtain a DG version of Koszul duality.
Besides these, some important classes of DG algebras are homologically smooth. For example, Calabi-Yau DG algebras introduced by Ginzburg in \cite{Gin} are homologically smooth by definition.  Recall that a connected cochain DG algebra $A$ is called an $n$-Calabi-Yau DG algebra, if $A$ is homologically smooth and $$R\Hom_{A^e}(A, A^e)\cong
\Sigma^{-n}A$$ in the derived category $\mathscr{D}((A^e)^{op})$ of right DG $A^e$-modules (cf. \cite{Gin,VdB}).  Especially, non-trivial Noetherian DG down-up algebras and DG polynomial algebras are  Calabi-Yau DG algebras by \cite{MHLX} and \cite{MGYC}, respectively.

 The motivation of this paper is to study the derived Picard group of a homologically smooth Koszul connected cochain DG algebra.
We show the following theorem (see Theorem \ref{dpd}). \\
\begin{bfseries}
Theorem \ A.
\end{bfseries}
Let $A$ be a homologically smooth Koszul DG algebra with Ext-algebra $\mathcal{E}$. Then we have a group isomorphism $$\mathrm{DPic}(A)\cong \mathrm{DPic}(\mathcal{E}).$$

Each DG algebra is a graded algebra with a differential
satisfying Leibniz rule.  Hence its properties are determined by the
joint effects of this two intrinsic structures. In general, the derived Picard group of a DG algebra is much more complicated than that of a finite dimensional local algebra. Fortunately,
the derived Picard group of a finite dimensional local algebra $\mathcal{E}$ is isomorphic to $\Bbb{Z}\times \mathrm{Pic}_k (\mathcal{E})$ (cf. \cite{Zim,RZ}), and it is a classical fact that
$\mathrm{Pic}_k (\mathcal{E})\cong \mathrm{Out}_k(\mathcal{E})$, which is attributed to Fr\"{o}hlich \cite{Fro} and can also be found in \cite[Theorem 37.16]{Rei}.
Hence
Theorem A gives an efficient way to compute the derived Picard group of a Koszul, homologically smooth connected cochain DG algebra.

Let $(A,\partial_A)$ be a connected cochain DG algebra.
For any cocycle element $z\in \mathrm{ker}(\partial_{A}^i)$, we write $\lceil z \rceil$ as the cohomology class in $H(A)$ represented by $z$.
Applying Theorem A, we compute in Section \ref{gradedreg} the derived Picard group of $A$, when $H(A)$ belongs to any one of the following $4$ cases:
\begin{enumerate}
\item $H(A)=k[\lceil x\rceil ], x\in \mathrm{ker}(\partial_A^1)$; \\
\item $H(A)=k\langle \lceil x_1\rceil, \cdots, \lceil x_n\rceil \rangle, x_1,\cdots, x_n\in \mathrm{ker}(\partial_A^1)$;                                                 \\
\item $H(A)=\k[ \lceil x_1\rceil,\lceil x_2\rceil], x_1,x_2\in \mathrm{ker}(\partial_A^1)$; \\
\item $H(A)=\k\langle \lceil x_1\rceil,\lceil z_2\rceil\rangle/(\lceil x_1\rceil \lceil x_2\rceil+\lceil x_2\rceil \lceil x_1\rceil ), x_1,x_2\in \mathrm{ker}(\partial_A^1)$.
\end{enumerate}
One sees that $H(A)$ is a Koszul regular graded algebra in any one of the four cases listed above. In the set of homologically smooth Koszul connected cochain DG algebras, there are DG algebras, whose cohomology graded algebras are neither Koszul nor regular. One can see such examples from \cite[Example 2.11]{MW2}, \cite[Proposition 6.5]{MXYA} and \cite[Proposition 6.1]{MHLX}. In Section \ref{special}, we compute the derived Picard groups of the following list of DG algebras:
\begin{center}
\begin{tabular}{|l|l|l|}
  \hline
  No &\quad\quad\quad\quad $A^{\#}$ & \quad\quad\quad\quad $\partial_{A}$ \\ \hline
  $1$ & $k\langle x_1,x_2\rangle, |x_1|=|x_2|=1$ &  $\begin{cases}\partial_{A}(x_1)=x_2^2\\
                                                                  \partial_{A}(x_2)=0
                                                                  \end{cases}$\\
  $2$ & $k\langle x_1,x_2\rangle, |x_1|=|x_2|=1$ &   $\begin{cases}\partial_{A}(x_1)=x_1x_2+x_2x_1-x_1^2-x_2^2\\
                                                                  \partial_{A}(x_2)=x_1x_2+x_2x_1-x_1^2-x_2^2
                                                        \end{cases}$\\
  $3$ & $k\langle x_1,x_2\rangle/(x_1x_2+x_2x_1), |x_1|=|x_2|=1$ &  $\begin{cases}\partial_{A}(x_1)=x_2^2\\
                                                                  \partial_{A}(x_2)=0
                                                                  \end{cases}$\\
  $4$ & $k\langle x_1,x_2\rangle/(x_1x_2+x_2x_1), |x_1|=|x_2|=1$ &  $\begin{cases}\partial_{A}(x_1)=x_1^2+x_2^2\\
                                                                  \partial_{A}(x_2)=x_1^2+x_2^2
                                                                  \end{cases}$\\
  $5$ & $\frac{k\langle x,y\rangle}{\begin{small}\left(\begin{array}{c}
   x^2y+(1-\xi)xyx-\xi yx^2  \\
   xy^2+(1-\xi)yxy-\xi y^2x  \\
\end{array}\right)\end{small}},|x|=|y|=1$ &  $\begin{cases}\partial_{A}(x)=y^2\\
                                                                  \partial_{A}(y)=0\\
                                                                  \end{cases}$\\
  \hline
\end{tabular}.
\end{center}

\section{Notation and conventions}
We assume that the reader is familiar with basics on differential graded homological algebra. If this is not the case,  we refer
to \cite{FHT,Kel2,HMS,MW2} for more details. Throughout this paper, $k$ will denote a fixed algebraically closed field.
For any $k$-vector space $V$,  we write $V^*=\Hom_{k}(V,k)$. For any $i\in \Bbb{Z}$ and graded $k$-space $V$, we denote $\Sigma^i V$ its $i$-th suspension, which is the graded $k$-vector space
$\Sigma^i M$ defined by $(\Sigma^iM)^j = M^{j+i}$. We set $\Sigma M = \Sigma^1 M$.
\subsection{DG algebras and DG coalgebras}
A cochain DG algebra $A$ over $k$ is a DG $k$-algebra $A=\bigoplus_{n\in
\Bbb Z}A^n$ with a differential $\partial_A$ of degree $1$. For any cochain DG algebra $A$, the product of its opposite algebra $A^{op}$ is defined by $a_1\diamond a_2 =(-1)^{|a_1|\cdot |a_2|}a_2a_1$.
We denote by $A^e = A\otimes A\!^{op}$ the enveloping DG algebra of $A$.
   Obviously, any DG right $A$-module can be identified as a DG left $A^{op}$-module.
An cochain DG algebra $A$ is called augmented if there is a DG algebra morphism $\varepsilon_A: A\to k$ with
$\varepsilon_A\circ\eta_A=\mathrm{id}_{k}$, where $\eta_A:k\to A$ is the unit map.   A cochain DG algebra $A$ is called  connected if  its underlying graded algebra is a connected graded algebra, i.e., $A^0=k$ and $A^i=0,$ for any $i<0$. One sees that any connected cochain DG algebra $A$ is augmented with
a canonical augmentation map
$\varepsilon: A\to A/A^{\ge 1}=k$. For the rest of this paper, we denote by $A$ a
connected cochain DG algebra over a field $k$ if no special assumption is
emphasized. The underlying graded algebra of $A$ is written by $A^{\#}$. For any DG $A$-module $M$,  we write $M^{\#}$ as its underlying graded $A^{\#}$-module.
A cochain DG coalgebra is graded coalgebra
$C=\bigoplus_{n\in \Bbb Z}C^n$ with a differential $\partial_C$ of degree $1$ such that $\Delta_C:C\to C\otimes C$ is a chain map.
Let $C$ be a cochain DG coalgebra with a counit $\varepsilon_C$. If there is a
coaugmentation map $\eta_C:k\longrightarrow C$ such that
$\partial_C\circ\eta_C=0$ and $\varepsilon_C\circ\eta_C=\mathrm{id}_{k}$, then $C$ is called coaugmented.

Let $A$ and $C$ be an augmented DG algebra and a coaugmented DG coalgebra, respectively. We denote by $B(A)$ and $\Omega(C)$ the bar construction and the cobar contruction of $A$ and $C$, respectively.
Let $M$ and $N$ be a right DG $A$-module and a DG right $C$-comodule, respectively. We write $B(M; A)=M\otimes B(A)$ and $\Omega(N; A)=N\otimes \Omega(C)$ for $M's$ bar construction  and $N's$ cobar construction, respectively. We
refer to \cite{Avr}, \cite{FHT} and \cite{HM} for detailed definitions of bar constructions and cobar constructions.
For $C$, it
has a decomposition $C=k\oplus \overline{C}$, in which $\overline{C}$
is the cokernel of $\eta_C$. There is a coproduct
$\overline{\Delta}:\overline{C}\to \overline{C}\otimes \overline{C}$
induced by $\Delta$ such that $(\overline{C},\overline{\Delta})$ is
a coalgebra without counit. A coaugmented DG coalgebra $C$ is called
{\it cocomplete} if, for any graded element $x\in \overline{C}$,
there is an integer $n$ such that
$\overline{\Delta}^n(x)=(\overline{\Delta}\otimes \mathrm{id}^{\otimes
n-1})\circ\cdots\circ(\overline{\Delta}\otimes
\mathrm{id})\circ\overline{\Delta}(x)=0$.
Let
$\varepsilon_A: A\to k$ and $\eta_C:k\to C$ be the
augmentation map of $A$  and the coaugmentation map of $C$, respectively.
\subsection{Twisted tensor product}
Let $A$ and $C$ be an augmented DG algebra and a coaugmented DG coalgebra, respectively.
A graded linear
map $\tau:C\to A$ of degree $1$ is called a {\it twisting cochain}
(\cite{HMS,Lef,LV,Avr}) if
$$\begin{array}{c}
\varepsilon_A\circ\tau\circ\eta_C=0,\ \text{and}\\
 m_A\circ(\tau\otimes
\tau)\circ\Delta_C+\partial_A\circ\tau+\tau\circ \partial_C=0,\end{array}$$
where $m_A$ is the multiplication operation of $A$.
Let $\tau: C\to A$ be a twisting cochain.  For any DG right $C$-comodule $N$,
the {\it twisted tensor product} (\cite{Lef,LV,Avr}) $N{}_{\tau}\otimes A$
is defined to be a DG right $A$-module:

(i) as a vector space $N{}_{\tau}\otimes A=N\otimes A$;

(ii) the differential is $$\delta(n\otimes a)=\partial_N(n)\otimes a+(-1)^{|n|}n\otimes
\partial_A(a)+\s_{(n)}(-1)^{|n_0|}n_{(0)}\otimes\tau(n_{(1)})a,$$ for graded
elements $n\in N$ and $a\in A$. Here, we use Sweedler's notation: $$\Delta_N(n)=\s_{(n)}n_{(0)}\otimes n_{(1)}.$$
For any DG left $C$-module $N'$, one can similarly define a DG left $A$-module $A\otimes_{\tau}N'$.

Dually, for a DG right $A$-module $M$, the
twisted tensor product $M\otimes_\tau C$ is defined to be a DG right $C$-comodule:

(i) as a vector space $M\otimes_\tau C=M\otimes C$;

(ii) the differential is $$\delta(m\otimes c)=\partial_M(m)\otimes c+(-1)^{|m|}m\otimes
\partial_C(c)-\s_{(c)}(-1)^{|m|}m\tau(c_{(1)})\otimes c_{(2)},$$ for
graded elements $m\in M$ and $c\in C$ with $\Delta_C(c)=\s_{(c)}c_{(1)}\otimes c_{(2)}$. For any DG left $A$-module $M'$, one can define a DG left $C$-comodule $C{}_{\tau}\otimes M'$ in a similar way.

Note the position of the subscript ``$\tau$" in the twisted tensor product.
Here, we emphasis that we take the same symbols on twisted tensor product as in \cite{Avr}.
A twisting map $\tau: C\to A$ is called {\it acyclic} if the map
$$\varepsilon^{ACA}: A\otimes{}_{\tau}C_{\tau}\otimes A\stackrel{\mathrm{id}\otimes \varepsilon_C\otimes \mathrm{id}}{\longrightarrow} A\otimes k\otimes A=A\otimes A\stackrel{\mu_A}{\longrightarrow} A$$
given by $\varepsilon^{ACA}(a\otimes c\otimes a')=\varepsilon_C(c)aa'$ is a quasi-isomorphism.

\subsection{Derived category of DG modules}
The category of DG left $A$-modules is denoted by $\mathscr{C}(A)$ whose
morphisms are DG morphisms. The homotopy category $\mathscr{H}(A)$ is the quotient category of
$\mathscr{C}(A)$, whose objects are the same as those of
$\mathscr{C}(A)$ and whose morphisms are the homotopy equivalence
classes of morphisms in $\mathscr{C}(A)$(cf.\cite{Kel2}). The derived category of DG $A$-modules is denoted by
$\mathscr{D}(A)$, which is constructed from the category
$\mathscr{C}(A)$ by inverting quasi-isomorphisms
(\cite{We},\cite{KM}).
The right derived functor of $\Hom$, is denoted by $R\Hom$, and the
left derived functor of $\otimes$, is denoted by $\,\otimes^L$. For any $M,N\in \mathscr{D}(A)$, $Q\in \mathscr{D}(A^{op})$ and $i\in \Bbb{Z}$, we write $$\mathrm{Tor}_A^i(Q,M)=H^i(Q\otimes^L_A M)\quad \text{and}\quad \mathrm{Ext}_A^i(M,N)=H^i(R\Hom_A(M,N)).$$
The notation $\mathscr{D}^*(A) (* = +,-, b)$ stands for the derived category of
bounded below (resp. bounded above, bounded) DG left $A$-modules.  According to the definition of compact object (cf. \cite{Kr1,Kr2})
in a triangulated category with arbitrary coproduct, a DG $A$-module $M$
is called compact, if
$\Hom_{\mathscr{D}(A)}(M,-)$ preserves all coproducts in
$\mathscr{D}(A)$. A DG $A$-module is compact if and only if it is
in the smallest triangulated thick subcategory of $\mathscr{D}(A)$
containing ${}_AA$ (see \cite[Theorem 5.3]{Kel2}).  The full subcategory
of $\mathscr{D}(A)$ consisting of compact DG $A$-modules is denoted
by $\mathscr{D}^c(A)$. For any augmented DG algebra $A$, we define its {\it Ext-algebra} by $$E=H(R\Hom_A(k,k))=\bigoplus\limits_{i\in \Bbb{Z}}\mathrm{Ext}^i_A(k,k).$$ We say that a connected cochain DG algebra $A$ a homologically smooth DG algebra if its Ext-algebra is finite dimension, or equivalently $A$ is compact as an DG $A^e$-module.  In DG homological algebra,
homologically smooth DG algebras play a similar role as regular
ring do in classical homological ring theory.

\subsection{Notations on (co)-derived categories}
Let $M$ be a DG right comodule over a coaugmented cocomplete DG coalgebra $C$. We have a
composition $\overline{\rho}: M\overset{\rho}\longrightarrow M\otimes
C\overset{\mathrm{id}\otimes \pi}\longrightarrow M\otimes \overline{C}.$
We say that $M$ is cocomplete if, for any homogeneous element $m\in M$,
there is an integer $n$ such that
$$\overline{\rho}^n(m)=(\overline{\rho}\otimes \mathrm{id}^{\otimes
n-1})\circ\cdots\circ(\overline{\rho}\otimes
\mathrm{id})\circ\overline{\rho}(m)=0.$$
If $C$ is finite dimensional, then any DG right $C$-comodule $M$ is cocomplete.
We write $\mathrm{DGcocom}\,C$ as the category of cocomplete DG right $C$-comodules, where the morphisms between objects are
DG morphisms of DG right $C$-comodules. Given a twisting cochain $\tau:C \to A$, we have a pair of adjoint functors $(L,R)$(cf. \cite{Lef,Kel3}):
$$L=-{}_{\tau}\otimes A:\text{DGcocom}C\rightleftarrows\mathscr{C}(A^{op}):
R=-\otimes{}_\tau C.$$
Let $M,N$ be cocomplete DG right $C$-comodules.
A DG comodule morphism $f$ from $M$ to $N$ is called a weak equivalences related to $\tau$,  if $Lf:LM\to LN$
is a quasi-isomorphism.  Let $\mathscr{H}(\mathrm{DGcocom}\,C)$ (cf.\cite{HW}) be the homotopy
category of $\mathrm{DGcocom}\,C$, and $\mathcal{W}$ be the class of weak
equivalences in the category  $\mathscr{H}(\mathrm{DGcocom}\,C)$. Then $\mathcal{W}$ is
a multiplicative system. The coderived category
$\mathscr{D}(\mathrm{DGcocom}\,C)$ of $C$ is defined to be
$\mathscr{H}(\mathrm{DGcocom}\,C)[\mathcal{W}^{-1}]$, i.e., the localization of
$\mathscr{H}(\mathrm{DGcocom}\,C)$ at the weak equivalence class $\mathcal{W}$
(cf. \cite{Kel3,Lef}).
With the natural exact triangles,
$\mathscr{D}(\mathrm{DGcocom}\,C)$ is a triangulated category.
The notation $\mathscr{D}^*(\mathrm{DGcocom}\,C) (* = +,-, b)$ stands for the coderived category of
bounded below (resp. bounded above, bounded) cocomplete DG right $C$-comodules.
\subsection{Notations on triangulated categories}
Let $\mathcal{C}$ be a given subcategory or simply a set of some
objects in a triangulated category $\mathcal{T}$. We denote by
$\mathrm{smd}(\mathcal{C})$ the minimal strictly full subcategory
which contains $\mathcal{C}$ and is closed under taking
direct summands. And we write
 $\mathrm{add}(\mathcal{C})$ as the intersection of all
strict and full subcategories of $\mathcal{T}$ that contain
$\mathcal{C}$ and are closed under finite direct
sums and all suspensions. Let $\mathcal{A}$ and $\mathcal{B}$ be
two strict and full subcategories of $\mathcal{T}$. Define
$\mathcal{A}\star \mathcal{B}$ as a full subcategory of
$\mathcal{T}$,  whose objects are described as follows: $M\in
\mathcal{A}\star \mathcal{B}$ if and only if there is an exact
triangle
$$ L\to M\to N\to \Sigma L, $$
where $L\in \mathcal{A}$ and $N\in \mathcal{B}$. For any strict and
full subcategory $\mathcal{C}$ of $\mathcal{T}$, one has
$\mathcal{A}\star (\mathcal{B}\star \mathcal{C}) = (\mathcal{A}\star
\mathcal{B})\star \mathcal{C}$ (see \cite{BBD} or \cite[1.3.10]{BV}). Thus, the
following notation is unambiguous:
\begin{equation*}
\mathcal{A}^{\star n} =
\begin{cases} 0\quad&\text{for}\, n =0;\\
\mathcal{A}\quad & \text{for}\, n =1; \\
\overbrace{\mathcal{A}\star\cdots\star
\mathcal{A}}^{n\,\text{copies}} \quad&\text{for}\, n\ge 1.
\end{cases}
\end{equation*}
We refer to the objects of $\mathcal{A}^{\star n}$ as $(n-1)$-fold
extensions of objects from $\mathcal{A}$. Define
$\mathcal{A}\diamond\mathcal{B} =
\mathrm{smd}(\mathcal{A}\star\mathcal{B})$. Let $\mathcal{E}$ be a
full subcategory of $\mathcal{T}$. Inductively, we define
$\langle\mathcal{E}\rangle_1 =
\mathrm{smd}(\mathrm{add}(\mathcal{E}))$ and
$\langle\mathcal{E}\rangle_n =
\langle\mathcal{E}\rangle_{n-1}\diamond \langle\mathcal{E}\rangle_1,
n\ge 2$.  We have the
associativity of $\diamond$ and the formula
$\mathcal{A}_1\diamond\mathcal{A}_2\diamond\cdots\diamond
\mathcal{A}_n
=\mathrm{smd}(\mathcal{A}_1\star\cdots\star\mathcal{A}_n)$ (see
\cite[Section 2]{BV}). Clearly, $\langle \mathcal{E} \rangle= \bigcup_n \langle \mathcal{E}\rangle_n $ is the smallest full triangulated subcategory of $\mathcal{T}$ containing $\mathcal{E}$ and closed
under forming direct summands.

\section{tilting dg modules and derived picard groups}
 In this section, we consider the two sided tilting DG modules over $A$ and the derived Picard group of $A$. The following definition is analogous to the
notion of two sided tilting complex, which is introduced independently by Rouquier-Zimermann \cite{RZ,Zim} and Yekutieli \cite{Yek1} based on Rickard's derived Morita theory \cite{Ric1}.

\begin{defn}{\rm
 A DG $A^e$-module $X$ is called tilting if there is a DG $A^e$-module $Y$ such that $X\otimes_A^L Y\cong A$ and $Y\otimes_A^LX\cong A$ in $\mathscr{D}(A^e)$. The DG $A^e$-module $Y$ is called a quasi-inverse of $X$.
}
\end{defn}

\begin{rem}
It is easy to see that the quasi-inverse of a tilting DG module is also tilting. And the quasi-inverse of a given tilting DG $A^e$-module is unique up to isomorphism in $\mathscr{D}(A^e)$.
If $X_1$ and $X_2$ are two tilting DG modules, then so is $X_1\otimes_A^L X_2$ by the associativity of $-\otimes_A^L-$.
\end{rem}
For the characterizations of tilting DG $A^e$-modules, we have
the following proposition, which is analogous to \cite[Theorem 1.6]{Yek1} and \cite[Theorem 5.6]{Yek3}.

\begin{prop}\label{chartilt}
 Let $X$ be a DG $A^e$-module. The following are equivalent
\begin{enumerate}
\item $X$ is a tilting DG $A^e$-module.
\item The functors $X\otimes_A^L-$ and $-\otimes_A^L X$ are auto-equivalences of $\mathscr{D}(A^{e})$.
\item The functors $X\otimes_A^L-$ and $-\otimes_A^L X$ are auto-equivalences of $\mathscr{D}(A)$ and $\mathscr{D}(A^{op})$ respectively.
\item The functors $X\otimes_A^L-$ and $-\otimes_A^L X$ are auto-equivalences of $\mathscr{D}^c(A)$ and $\mathscr{D}^c(A^{op})$ respectively.
\item $\langle {}_AX\rangle =\mathscr{D}^c(A)$,  $\langle X_A\rangle = \mathscr{D}^c(A^{op})$,
 and the adjunction morphisms $$A\to R\Hom_A(X,X) \quad \text{and}  \quad A\to R\Hom_{A^{op}}(X,X)$$ in $\mathscr{D}^c(A^{e})$ are isomorphisms.
\end{enumerate}
\end{prop}

\begin{proof}
$(1)\Rightarrow (2)$
By the definition of tilting DG $A^e$-module, there is a DG $A^e$-module $Y$ such that $X\otimes_A^L Y\cong A$ and $Y\otimes_A^LX\cong A$ in $\mathscr{D}(A^{e})$.
Let $G$ and $H$ be two endofunctors of $\mathscr{D}(A^{e})$ defined by $G(M)= X\otimes_A^LM$ and $H(M)=Y\otimes_A^LM$ respectively. It is easy to see that $F\circ G= \mathrm{id}_{\mathscr{D}(A^{e})}$ and $G\circ F=\mathrm{id}_{\mathscr{D}(A^{e})}$. Hence $G=X\otimes_A^L-$ is an equivalence of $\mathscr{D}(A^{e})$. Similarly, we can get that $-\otimes_A^L X$ is also an equivalence of $\mathscr{D}(A^{e})$.

$(2)\Rightarrow (1)$ Since the functors $X\otimes_A^L-$ and $-\otimes_A^L X$ are equivalences of $\mathscr{D}(A^{e})$, they are both essentially surjective. So there are some $Y'$ and $Y''$ in $\mathscr{D}(A^{e})$ such that $X\otimes_A^L Y'\cong A$ and $Y''\otimes_A^L X\cong A$ in $\mathscr{D}(A^{e})$. By the associativity of $-\otimes_A^L-$, we have
$$Y'\cong Y''\otimes_A^L X\otimes_A^L Y'\cong Y''$$ in $\mathscr{D}(A^{e})$. Hence $X$ is an invertible DG $A^e$-module.

$(1)\Rightarrow (3)$  Let $Y$ be the quasi-inverse of $X$. Since $X\otimes_A^L Y\cong A$ and $Y\otimes_A^LX\cong A$ in $\mathscr{D}(A^{e})$, it is easy to see that $$(X\otimes_A^L-)\circ (Y\otimes_A^L-)=\mathrm{id}_{\mathscr{D}(A)}=(Y\otimes_A^L-)\circ (X\otimes_A^L-)$$ and $(-\otimes_A^L X)\circ (-\otimes_A^L Y)=\mathrm{id}_{\mathscr{D}(A)}=(-\otimes_A^L Y)\circ (-\otimes_A^L X)$. Hence the functors $X\otimes_A^L-$ and $-\otimes_A^L X$ are auto-equivalences of $\mathscr{D}(A)$ and $\mathscr{D}(A^{op})$ respectively.

$(3)\Rightarrow (5)$  Let $G$ and $H$ be the inverse functors of $X\otimes_A^L-$ and $-\otimes_A^L X$ respectively. We claim $G$ and $H$ commute with infinite direct sums.
Let $\{M_i|i\in I\}$ be a family of objects in $\mathscr{D}(A)$. For any $i\in I$,  there exists $N_i\in \mathscr{D}(A)$ such that $X\otimes_A^L N_i=M_i$. We have $G(M_i)= G\circ (X\otimes_A^L-)(N_i)=N_i$.  Hence $$G(\bigoplus\limits_{i\in I} M_i)=G(\bigoplus\limits_{i\in I}X\otimes_A^L N_i)=G\circ (X\otimes_A^L-) (\bigoplus\limits_{i\in I}N_i)= \bigoplus\limits_{i\in I}N_i=\bigoplus\limits_{i\in I}G(M_i).$$ Similarly, we can prove $H$ commutes with infinite direct sums. In order to prove that ${}_AX\in \mathscr{D}^c(A)$, lets consider the following commutative diagram
\begin{align*}
\xymatrix{\oplus_{i\in I}\Hom_{\mathscr{D}(A)}(X,M_i)
\ar[r]\ar[d]_{G} & \Hom_{\mathscr{D}(A)}(X,\oplus_{i\in I} M_i)
\ar[d]^{G}
\\
\oplus_{i\in I}\Hom_{\mathscr{D}(A)}(G(X),G(M_i)) \ar[r]
&\Hom_{\mathscr{D}(A)}(G(X),\oplus_{i\in I}G(M_i)).
\\}
\end{align*}
The vertical arrows are bijective since $G$ is an equivalence and the lower horizontal map is an isomorphism since ${}_AG(X)\cong {}_A G\circ (X\otimes_A^L-)(A)\cong {}_AA$  is compact in $\mathscr{D}(A)$. Hence the upper horizontal map is also bijective. This implies that ${}_AX$ is compact. We can similarly prove that $X_A$ is compact. Let $F_X$ be a semi-free resolution of the DG $A^e$-module $X$.
Recall that a semi-free resolution of a DG $A$-module $M$ is a semi-free DG $A$-module $F_M$ such that there exists a quasi-isomorphism from $F_M$ to $M$ (cf.\cite{FHT}).
 By \cite[Proposition 1.3]{FIJ}, $F_X$ is both a K-projective DG $A$-module and a K-projective DG $A^{op}$-module. Here, a DG $A$-module $P$ is called K-projective if the functor $\Hom_A(P,-)$ preserves quasi-isomorphisms (cf. \cite[1.2]{FIJ}).  It is easy to see that
the canonical morphism \begin{align*}
\theta: A & \to \Hom_{A}(F_X, F_X) \\
a& \mapsto [r_a: x\to xa]
\end{align*}
is a morphism of DG $A^e$-modules and it
is a quasi-isomorphism since the equivalence $X\otimes_A^L-: \mathscr{D}(\mathrm{DGmod}\,\,A)\to \mathscr{D}(A)$ induces the
  bijective map
  $$
   H^i(\theta):
\Hom_{\mathscr{D}(A)}(\Sigma^i A,A)\longrightarrow \Hom_{\mathscr{D}(A)}(\Sigma^iX,X)\cong  H^i(\Hom_A(F_X,F_X)),$$
 $\forall i\in \Bbb{Z}$. Hence $A\cong R\Hom_{A}(X,X)$ in $\mathscr{D}(A^{e})$. Similarly,  $ R\Hom_{A^{op}}(X,X)\cong A$ in $\mathscr{D}(A^{e})$.
Since ${}_AX\in \mathscr{D}^c(A)$ and $X_A\in \mathscr{D}^c(A^{op})$,  any objects in $\langle {}_AX \rangle $ and $\langle X_A\rangle$ are compact.
To show that $$\langle {}_AX\rangle = \mathscr{D}^c(A)\quad \text{and}\quad \langle X_A\rangle =\mathscr{D}^c(A^{op}),$$ it suffices to prove that ${}_AA \in \langle{}_AX\rangle$ and $A_A\in \langle X_A\rangle$. Since $A\cong R\Hom_{A}(X,X)$ and $ R\Hom_{A^{op}}(X,X)\cong A$ in $\mathscr{D}(A^{e})$, we have $${}_AA\cong {}_AR\Hom_{A^{op}}(X,X\otimes_AA)\cong X\otimes_A^LR\Hom_{A^{op}}(X,A)$$ and $$A_A\cong R\Hom_A(X,A\otimes_A X)_A\cong R\Hom_A(X,A)\otimes_A^L X. $$  By \cite[Proposition 3.6]{MW2}, one sees that
$$R\Hom_{A^{op}}(X,A)\in \mathscr{D}^c(A)\,\,\text{and}\,\, R\Hom_A(X,A)\in \mathscr{D}^c(A^{op}).$$
 So ${}_AA\cong X\otimes_A^LR\Hom_{A^{op}}(X,A)\in \langle {}_AX\rangle$  and $ A_A\cong R\Hom_A(X,A)\otimes_A^L X \in \langle X_A\rangle.$

$(5) \Rightarrow (1)$ In $\mathscr{D}(A^{e})$, we have
$ R\Hom_A(X,X)\cong R\Hom_A(X,A)\otimes_A^L X$ and $$R\Hom_{A^{op}}(X,X) \cong X\otimes_A^LR\Hom_{A^{op}}(X,A),$$ since ${}_AX$ and $X_A$ are both compact. Therefore, $$R\Hom_A(X,A)\otimes_A^L X\cong A\,\, \text{and}\,\, X\otimes_A^LR\Hom_{A^{op}}(X,A)\cong A.$$
By the associativity of $-\otimes_A^L-$,  we get
$$R\Hom_A(X,A)\cong R\Hom_A(X,A)\otimes_A^L X \otimes_A^L R\Hom_{A^{op}}(X,A)\cong R\Hom_{A^{op}}(X,A)$$
in $\mathscr{D}(A^{e})$. So the DG $A^e$-module $X$ is tilting.

$(4)\Rightarrow (5)$  We have ${}_AX=X\otimes_A^L A\in \mathscr{D}^c(\mathrm{DGmod}\,\,A)$ since $X\otimes_A^L-$ is an auto-equivalence of $\mathscr{D}^c(A)$. Similarly,  $X_A\in \mathscr{D}^c(A^{op})$.  Let $F$ be a semi-free resolution of the DG $A^e$-module $X$.  By \cite[Proposition 1.3]{FIJ}, $F$ is both a K-projective DG $A$-module and a K-projective DG $A^{op}$-module. It is easy to see that
the canonical morphism \begin{align*}
\beta: A & \to \Hom_{A}(F, F) \\
a& \mapsto [r_a: x\to xa]
\end{align*}
is a morphism of DG $A^e$-modules and it
is a quasi-isomorphism since the equivalence $X\otimes_A^L-: \mathscr{D}^c(A)\to \mathscr{D}^c(A)$ induces the
  bijective map
$$H^i(\beta): \Hom_{\mathscr{D}^c(A)}(\Sigma^i A,A)\to \Hom_{\mathscr{D}^c(A)}(\Sigma^iX,X)\cong H^i(\Hom_A(F,F))$$
for any $i\in \Bbb{Z}$. Hence $A\cong R\Hom_{A}(X,X)$ in $\mathscr{D}(A^{e})$. Similarly, we can get that  $ R\Hom_{A^{op}}(X,X)\cong A$ in $\mathscr{D}(A^{e})$.

$(5)\Rightarrow (4)$ Since ${}_AX\in \mathscr{D}^c(A)$ and $X_A\in \mathscr{D}^c(A^{op})$, it is easy to see that $X\otimes_A^L-$ and $-\otimes_A^L X$ are endofunctors of $\mathscr{D}^c(A)$ and $\mathscr{D}^c(A^{op})$ respectively. For any $M\in \mathscr{D}^c(A)$, we have
 $${}_AR\Hom_A(X,M)\in \langle {}_AR\Hom_A(X,X) \rangle =\langle {}_AA\rangle $$ since $\langle{}_AX\rangle=\mathscr{D}^c(A)$ and $R\Hom_A(X,X)\cong A$ in $\mathscr{D}(A^{e})$.
Hence the functor $R\Hom_A(X,-)$ is an endofunctor of $\mathscr{D}^c(A)$. Similarly, we can show that $R\Hom_{A^{op}}(X,-)$ is an endofunctor of $\mathscr{D}^c(A^{op})$.
Since the adjunction morphism $A\to R\Hom_A(X,X)$ in $\mathrm{D}(A^e)$ is an isomorphism, we have $$R\Hom_A(X,X\otimes_A^L M)\cong R\Hom_A(X,X)\otimes_A^L M\cong M$$ and $$X\otimes_A^LR\Hom_A(X,M) \cong R\Hom_A(X,X\otimes_A^LM)\cong R\Hom_A(X,X)\otimes_A^LM\cong M$$ in $\mathrm{D}(A)$, for any object $M\in \mathscr{D}^c(A)$. Hence the functor $$X\otimes_A^L-: \mathscr{D}^c(A)\to \mathscr{D}^c(A)$$ is an equivalence with inverse $\RHom_{A}(X,-)$. Similarly, $$-\otimes_A^L X: \mathscr{D}^c(A^{op})\to \mathscr{D}^c(A^{op})$$ is an equivalence with inverse $\RHom_{A^{op}}(X,-)$.
\end{proof}
\begin{defn}{\rm For a connected cochain DG algebra $A$,
we define its derived Picard group as the abelian group $\mathrm{DPic}(A)$, whose elements are the isomorphism classes of tilting DG $A^e$-modules in $\mathscr{D}(A^{e})$. In $\mathrm{DPic}(A)$,  the product of the classes of $X$ and $Y$ is given by the class of $X\otimes_A^L Y$, and the unit element is the class of $A$. }
\end{defn}

\section{some useful facts}
In this section, we list some important statements on a connected cochain DG algebra $A$ and DG $A$-modules.
 Let $F$ be a semi-free DG $A$-module. It is called minimal if $\partial_F(F)\subseteq A^{\ge 1}F$. For any DG $A$-module $M$, a minimal semi-free resolution of $M$ is a minimal semi-free DG $A$-module together with a quasi-isomorphism $\eta: F\to M$.  Sometimes, we just say $F$ is a minimal semi-free resolution of $M$ for briefness. As to its existence, we have the following lemma.
\begin{lem}\cite[Proposition 2.4]{MW1}
Any  DG $A$-module $M$ in $\mathscr{D}^+(A)$ admits a minimal semi-free resolution.
\end{lem}
Especially, we have the following characterization of a compact DG $A$-module in terms of minimal semi-free resolution.
\begin{lem}\cite[Proposition 3.3]{MW1}
For any DG $A$-module $M$, it is compact if and only if it admits a minimal semi-free resolution $F_M$, which has a finite semi-basis.
\end{lem}

 A connected cochain DG algebra $A$ is called homologically smooth if $A$ as an $A^e$-module is compact.  In DG homological algebra,
homologically smooth DG algebras play a similar role as regular
ring do in classical homological ring theory. We have the following lemma.
\begin{lem} \cite[Corollary 2.7]{MW3}
The connected cochain DG algebra $A$ is homologically smooth if and only if
$k\in \mathscr{D}^c(A)$.
\end{lem}
In \cite{HW},  the third author and Wu introduce the concept of
Koszul DG algebras. By the definition,  $A$ is called Koszul if ${}_Ak$ admits a minimal semi-free resolution, which has a semi-basis concentrated in degree $0$.  When $A$ is a homologically smooth Koszul DG algebra, we have the following result.
\begin{lem}\cite[Lemma 9.2]{MW2}\label{local}
Let $A$ be a homologically smooth Koszul DG algebra. Then its Ext-algebra $E$ is a finite dimensional local algebra concentrated in degree $0$.
\end{lem}


\section{proof of theorem a}
In this section, we want to prove Theorem A, which gives a
 shortcut to  compute the derived Picard group of homologically smooth  Koszul  connected DG algebras.
To prove Theorem A, we still need some preparation. The following lemma is proved by Lef$\grave{e}$vre in \cite[Ch.2]{Lef}, and also can be found in \cite[Theorem 4.1]{HW}.
\begin{lem}\label{codeq}  Let $C$ be a cocomplete DG coalgebra, $A$ an augmented DG
algebra and $\tau:C\to A$ is a twisting cochain. The following are
equivalent
\begin{itemize}
    \item [(i)] The map $\tau$ induces a quasi-isomorphism $\Omega(C)\to
    A$;
    \item [(ii)] The map $\varepsilon^{ACA}:A\otimes_\tau C_{\tau} \otimes A\to A$ is a
    quasi-isomorphism;
    \item [(iii)] The functors $L$ and $R$ induce equivalences of triangulated categories (also denoted by $L$ and $R$)
     $$L:\mathscr{D}(\mathrm{DGcocom}\,C)\rightleftarrows\mathscr{D}(A^{op}):R.\quad \square$$
\end{itemize}
\end{lem}

\begin{lem}\label{important}
Let $E$ be a local finite dimensional $k$-algebra with a residue field $k$. We regard $E$ as a DG algebra concentrated in degree $0$. This makes
$E^*$ a DG coalgebra concentrated in degree $0$.
Assume that $\tau: E^*\to A$ is a twisting cochain such that $\tau$ induces a quasi-isomorphism of DG algebras $\Omega(E^*)\to A$. Then $E^*{}_{\tau}\otimes A\otimes_{\tau}E^*$ is quasi-isomorphic to $E^*$ as an $E$-bimodule.
\end{lem}
\begin{proof}
By \cite[4.2]{Avr}, the the composition map
$$\eta^{E^*AE^*}:E^*\stackrel{\mu_{E}^*}{\longrightarrow}E^*\otimes E^*=E^*\otimes \k\otimes E^*\stackrel{\mathrm{id}_{E^*}\otimes \eta_A\otimes\mathrm{id}_{E^*}}{\longrightarrow}E^*{}_{\tau}\otimes A\otimes_{\tau} E^*$$ is a morphism of both left and right DG $E^*$-comodules. So $\eta^{E^*AE^*}$ is a chain map.
We claim that $\eta^{E^*AE^*}$ is also a morphism of DG $E$-bimodules. To see this, we only need to show that $\mu_{E}^*$ is a morphism of $E$-bimodule by the definition of $\eta^{E^*AE^*}$. One sees that $\mu_E:{}_EE\otimes E_E\to {}_EE_E$ defined by $\mu_E(e\otimes e')=ee'$ is a morphism of $E$-bimodules since
we have
$$
\mu_E[(e\otimes e')(a\otimes a')]= \mu_E(ea\otimes a'e')=eaa'e'=(e\otimes e')(aa')=(e\otimes e')\cdot\mu_E(a\otimes a'),
$$
for any $e\otimes e', a\otimes a'\in {}_EE\otimes E_E$. One sees that $$\mu_E^*:({}_EE_E)^*\to \Hom_k({}_EE\otimes E_E,k)\cong (E_E)^*\otimes ({}_EE)^*$$ is also a morphism of $E$-bimodules since
$$
\mu_E^*[(e'\otimes e)\cdot f](a'\otimes a)=[(e'\otimes e)\cdot f](aa')=f(eaa'e')
$$
and
$$
[(e'\otimes e)\cdot \mu_E^*(f)](a'\otimes a)=\mu_E^*(f)(a'e'\otimes ea )=f(eaa'e'),$$ for any $a'\otimes a, e'\otimes e\in E_E\otimes {}_EE$.
Since $\tau$ induces a quasi-isomorphism $\Omega(E^*)\to A$ of DG algebras, one sees that $\tau$ is acyclic
by the ``if" part of \cite[Theorem 4.1]{Avr}. Then the ``only if" part of \cite[Theorem 4.1]{Avr} implies that
is a quasi-isomorphism.
\end{proof}

\begin{thm}\label{dpd}
Let $A$ be a homologically smooth Koszul DG algebra with Ext-algebra $E$. Then we have a group isomorphism $$\mathrm{DPic}(A)\cong \mathrm{DPic}(E).$$
\end{thm}
\begin{proof}
By Lemma \ref{local}, $E$ is a finite dimensional local algebra concentrated in degree $0$.
So the
vector space dual $E^*$ is a coaugmented coalgebra which is of
course cocomplete. Hence all the DG $E^*$-comodules are cocomplete.
By \cite[4.5]{Avr}, we have a canonical twisting cochain $\tau: E^*\to
\Omega(E^*)$ satisfies the condition (i) in Lemma \ref{codeq}.
So we have equivalences
$$L:\mathscr{D}(\mathrm{DGcocom}\,E^*)\rightleftarrows \mathscr{D}(\Omega(E^*)^{op}):R.$$
Since $A$ is homologically smooth, we have ${}_Ak\in \mathscr{D}^c(A)$. By \cite[Proposition 4.2]{HW}, we have equivalences
$$L:\mathscr{D}^+(\mathrm{DGcocom}\,E^*)\rightleftarrows \mathscr{D}^+(\Omega(E^*)^{op}):R.$$
By \cite[Lemma 1.6.4]{Mon}, the category of left
$E$-modules is naturally equivalent to the category of right
$E^*$-comodules since $E$ is a finite dimensional algebra.  Hence we have equivalences
$$P:\mathscr{D}^+(E)\rightleftarrows \mathscr{D}^+(\mathrm{DGcocom}\,E^*):Q.$$
By \cite[Lemma 3.7]{HW}, there is a quasi-isomorphism $\varphi: \Omega(E^*)\to A$ of DG algebras. It induces the following equivalence of triangulated categories
\begin{align*}\label{finaldual}
\xymatrix{&\mathscr{D}^+(A^{op})\quad\quad\ar@<1ex>[r]^{\varphi^*}&\quad\quad
\mathscr{D}^+(\Omega(E^*)^{op})  \ar@<1ex>[l]^{-\otimes_{\Omega(E^*)}^LA}}.
\end{align*}
Let $\Phi=(-\otimes_{\Omega(E^*)}^LA)\circ L\circ P$ and $\Psi= Q\circ R\circ \varphi^*$.  We have the following equivalence of triangulated categories
\begin{align*}
\xymatrix{&\mathscr{D}^+(A^{op})\quad\quad\ar@<1ex>[r]^{\Psi}&\quad\quad
\mathscr{D}^+(E)  \ar@<1ex>[l]^{\Phi}}.
\end{align*}
By \cite[Corollary 4.6]{HW}, we
have an equivalence of triangulated categories
\begin{align*}
\xymatrix{&\mathscr{D}^c(A^{op})\quad\quad\ar@<1ex>[r]^{\Psi}&\quad\quad
\mathscr{D}^b(E)  \ar@<1ex>[l]^{\Phi}}
\end{align*}
when $\Phi$ and $\Psi$ are restricted to the full subcategories $\mathscr{D}^b(E)$ and $\mathscr{D}^c(A^{op})$ respectively.
For any tilting DG module $X$ of $A$, the functor $$-\otimes_A^L X: \mathscr{D}^c(A^{op})\to \mathscr{D}^c(A^{op})$$ is an
equivalence by Proposition \ref{chartilt}. One sees that $\Psi\circ (-\otimes_A^L X)\circ \Phi$ is a functor from $\mathscr{D}^b(E)$ to $\mathscr{D}^b(E)$.
Let $Y$ be the quasi-inverse of $X$. Then $X\otimes_{A}^LY\cong A$ and $Y\otimes_A^LX\cong A$ in $\mathscr{D}(A^e)$. We have
\begin{align*}
&\quad [\Psi\circ (-\otimes_A^L X)\circ \Phi]\circ [\Psi\circ (-\otimes_A^LY)\circ \Phi]  \\
=&[\Psi\circ (-\otimes_A^L X)\circ(-\otimes_{\Omega(E^*)}^LA)\circ L\circ P]\circ[Q\circ R\circ \varphi^*\circ (-\otimes_A^L Y)\circ\Phi]\\
\simeq & \Psi\circ (-\otimes_A^L X)\circ (-\otimes_A^L Y)\circ\Phi \\
= & \Psi \circ (-\otimes_A^L Y\otimes_A^L X)\circ\Phi \\
\cong &  \Psi \circ (-\otimes_A^L A)\circ\Phi \\
\cong &   \Psi \circ \Phi \\
\simeq & \mathrm{id}_{\mathscr{D}^b(E)}
\end{align*}
and
\begin{align*}
&\quad [\Psi\circ (-\otimes_A^L Y)\circ \Phi]\circ [\Psi\circ (-\otimes_A^LX)\circ \Phi]  \\
=&[\Psi\circ (-\otimes_A^L Y)\circ(-\otimes_{\Omega(E^*)}^LA)\circ L\circ P]\circ[Q\circ R\circ \varphi^*\circ (-\otimes_A^L X)\circ\Phi]\\
\simeq & \Psi\circ (-\otimes_A^L Y)\circ (-\otimes_A^L X)\circ\Phi \\
= & \Psi \circ (-\otimes_A^L X\otimes_A^L Y)\circ\Phi \\
\cong &  \Psi \circ (-\otimes_A^L A)\circ\Phi \\
\cong &   \Psi \circ \Phi \\
\simeq & \mathrm{id}_{\mathscr{D}^b(E)}.
\end{align*}
Hence $\Psi\circ (-\otimes_A^L X)\circ \Phi:\mathscr{D}^b(E)\to \mathscr{D}^b(E)$ is an equivalence.
By \cite[Corollary 1.9]{Yek1}, there is a tilting complex $T\in \mathscr{D}^b(E^e)$ with $${}_ET\cong \Psi\circ (-\otimes_A^L X)\circ \Phi(E) $$
in $\mathscr{D}(E)$. Hence we can define a map \begin{align*}
\lambda: \quad &\mathrm{DPic}(A)^{op} \to  \mathrm{DPic}(E)  \\
         & X\mapsto  \Psi\circ (-\otimes_A^L X)\circ \Phi(E).
\end{align*}
We claim that $\lambda$ is a group morphism.
 Let $\tau'=\varphi\circ \tau: E^*\to A$. Then $\tau'$ is a twisting cochain  satisfying  the condition $(i)$ in Theorem \ref{codeq},
since $\varphi$ is a quasi-isomorphism.
 By the definitions of $\Phi$ and $\Psi$,  we have the following isomorphisms in $\mathscr{D}(E^e)$
  \begin{align*}
\tag{Isom1} \Psi\circ (-\otimes_A^L X)\circ \Phi(E)&= \varphi^* [(E{}_{\tau}\otimes\Omega(E^*)\otimes^L_{\Omega(E^*)}A)\otimes_A^LX]\otimes_{\tau}E^* \\
&\cong (E{}_{\tau'}\otimes X)\otimes_{\tau'}E^*,
\end{align*}
where the DG left and DG right $E$-module structures on $(E_{\tau'}\otimes X)\otimes_{\tau'}E^*$ are inherited from $E^*$ and $E$, respectively.
Let $Y$ be another tilting DG module of $A$. As DG $A^e$-modules, $X$ and $Y$ admit K-projective resolutions $F_X$ and $F_Y$ respectively. By \cite[Proposition 1.3 (c)]{FIJ},
$F_X$ and $F_Y$ are K-projective as DG left (resp. right) $A$-modules.
Define
\begin{align*}
\theta: \quad & (E{}_{\tau'}\otimes F_X\otimes_{\tau'}E^*)\otimes_E(E{}_{\tau'}\otimes F_Y\otimes_{\tau'}E^*)\to E{}_{\tau'}\otimes F_Y\otimes_{\tau'}E^*{}_{\tau'}\otimes F_X\otimes_{\tau'} E^* \\
         &\,\,\, (e_1\otimes x\otimes f_1)\otimes (e_2\otimes y\otimes f_2) \mapsto (-1)^{|x|\cdot|y|}e_2\otimes y\otimes e_1\cdot f_2\otimes x\otimes f_1.
\end{align*}
Note that $E$ is concentrated in degree $0$. It is easy to check that $\theta$ is a chain map.
And $\theta$ is bijective with the following inverse chain map
\begin{align*}
\theta^{-1}: \quad & E{}_{\tau'}\otimes F_Y\otimes_{\tau'}E^*{}_{\tau'}\otimes F_X\otimes_{\tau'} E^*  \to  (E{}_{\tau'}\otimes F_X\otimes_{\tau'}E^*)\otimes_E(E{}_{\tau'}\otimes F_Y\otimes_{\tau'}E^*)\\
 &\quad\quad \,\, e\otimes y\otimes f\otimes x\otimes g\mapsto (-1)^{|x|\cdot|y|}(1\otimes x\otimes g)\otimes (e\otimes y\otimes f).
\end{align*}
Therefore,  we have the following isomorphism in $\mathscr{D}(E^e)$
\begin{align*}
\lambda(X\cdot Y)&=\lambda(Y\otimes_A^L X)\\
                             &=\lambda(F_Y\otimes_AF_X)\\
                             & = \Psi\circ (-\otimes_A F_Y\otimes_AF_X)\circ \Phi(E)\\
 &= \Psi\circ (-\otimes_A F_X)\circ (-\otimes_AF_Y)\circ \Phi(E)\\
 &\cong \Psi\circ (-\otimes_A F_X)\circ \Phi\circ \Psi\circ (-\otimes_A F_Y)\circ \Phi(E)\\
&\stackrel{(a)}{\cong}  \{[(E{}_{\tau'}\otimes F_Y)\otimes_{\tau'}E^*]{}_{\tau'}\otimes F_X\}\otimes_{\tau'}E^*\\
&=  E{}_{\tau'}\otimes F_Y\otimes_{\tau'}E^*{}_{\tau'}\otimes F_X\otimes_{\tau'} E^*\\
&\stackrel{(b)}{\cong} (E{}_{\tau'}\otimes F_X\otimes_{\tau'}E^*)\otimes_E(E{}_{\tau'}\otimes F_Y\otimes_{\tau'}E^*)\\
&\cong \lambda(F_X)\otimes_E \lambda(F_Y)\\
&\cong \lambda(X)\otimes_E^L \lambda(Y)= \lambda(X)\cdot \lambda(Y),
\end{align*}
where $(a)$ is obtained by using isomorphism (Isom1) twice, and $(b)$ is from the fact that $\theta$ is bijective.
Hence $\lambda$ is a group morphism. It remains to show that $\lambda$ is bijective.
For any tilting complex $T$ of $E$,  let $F_T$ be its K-projective resolution. We have
\begin{align*}
\tag{Isom2} \Phi\circ (T\otimes_E^L-)\circ \Psi(A) &= \{[T\otimes_E^L(\varphi^*(A)\otimes_{\tau}E^*)]{}_{\tau}\otimes\Omega(E^*)\}\otimes_{\Omega(E^*)}^LA\\
                                       &\cong [F_T\otimes_E(A\otimes_{\tau'} E^*)]{}_{\tau'}\otimes A \\
                                       &\cong [(A\otimes_{\tau'}E^*)\otimes_{E^{op}}F_T]{}_{\tau'}\otimes A.
\end{align*}
in $\mathscr{D}(A^{e})$.
We claim that $\Phi\circ (T\otimes_E^L-)\circ \Psi(A)$ is a tilting DG module of $A$. Indeed, if $W$ is the quasi-inverse tilting complex of $T$ and  $F_W$
is a K-projective resolution of $W$, then we have
\begin{align*}
&\quad\quad\quad [\Phi\circ (T\otimes_E^L-)\circ \Psi(A)]\otimes_{A}^L [\Phi\circ (W\otimes_E^L-)\circ \Psi(A)]  \\
&\stackrel{(c)}{\cong}\{[(A\otimes_{\tau'}E^*)\otimes_{E^{op}}F_T]{}_{\tau'}\otimes A\}\otimes_{A}^L \{[(A\otimes_{\tau'}E^*)\otimes_{E^{op}}F_W]{}_{\tau'}\otimes A\} \\
&\cong [(A\otimes_{\tau'}E^*)\otimes_{E^{op}}F_T]{}_{\tau'}\otimes \{[(A\otimes_{\tau'}E^*)\otimes_{E^{op}}F_W]{}_{\tau'}\otimes A\} \\
&\stackrel{(d)}{\cong} [(A\otimes_{\tau'}E^*)\otimes_{E^{op}}F_T]\square_{E^*}E^*{}_{\tau'}\otimes(A\otimes_{\tau'}E^*){}_{\tau'}\otimes_{E^{op}}F_W\otimes A \\
&\stackrel{(e)}{\cong} [(A\otimes_{\tau'}E^*)\otimes_{E^{op}}F_T]\square_{E^*}E^*{}_{\tau'}\otimes_{E^{op}}F_W\otimes A \\
&\stackrel{(f)}{\cong}  (A\otimes_{\tau'}E^*)\otimes_{E^{op}}F_T{}_{\tau'}\otimes_{E^{op}}F_W\otimes A \\
&\cong  (A\otimes_{\tau'}E^*)\otimes_{E^{op}}F_W\otimes_{E}F_T{}_{\tau'}\otimes A \\
&\cong  A\otimes_{\tau'}E^*\otimes_{E^{op}}E{}_{\tau'}\otimes A \\
&\cong A\otimes_{\tau'}E^*{}_{\tau'}\otimes A\\
&\stackrel{(g)}{\cong} A
\end{align*}
in $\mathscr{D}(A^{e})$, where the symbol ``$\square_{E^*}$" is the cotensor product over the coalgebra $E^*$ (cf. \cite[Definition 8.4.2]{Mon}), $(d)$ and $(f)$ are by \cite[Proposition 2.21, Proposition 2.3.6]{DNR}; $(c)$, $(e)$ and $(g)$ are obtained by (Isom2), Lemma \ref{important} and Lemma \ref{codeq}, respectively. Hence $\Phi\circ (T\otimes_E^L-)\circ \Psi(A)$ is a tilting DG module of $A$. So we have the following map
 \begin{align*}
\beta: \quad & \mathrm{DPic}(E)  \to \mathrm{DPic}(A)^{op}   \\
         & T\mapsto  \Phi\circ (T\otimes_E^L-)\circ \Psi(A).
\end{align*}
We have
\begin{align*}
\beta(T\otimes_E^LW)&= \beta(F_T\otimes_EF_W)\\
                    &\stackrel{(h)}{\cong}[(A\otimes_{\tau'}E^*)\otimes_{E^{op}}(F_T\otimes_{E}F_W)]{}_{\tau'}\otimes A\\
                    &\cong [(A\otimes_{\tau'}E^*)\otimes_{E^{op}}(F_W\otimes_{E^{op}}F_T)]{}_{\tau'}\otimes A\\
                    &\stackrel{(i)}{\cong} [(A\otimes_{\tau'}E^*)\otimes_{E^{op}}F_W]\square_{E^*}(E^*\otimes_{E^{op}}F_T){}_{\tau'}\otimes A \\
                    &\stackrel{(j)}{\cong} [(A\otimes_{\tau'}E^*)\otimes_{E^{op}}F_W]\square_{E^*}E^*{}_{\tau'}\otimes A\otimes_{\tau'}E^*{}_{\tau'}\otimes_{E^{op}}F_T\otimes A \\
                    &\stackrel{(k)}{\cong} [(A\otimes_{\tau'}E^*)\otimes_{E^{op}}F_W]{}_{\tau'}\otimes \{[A\otimes_{\tau'}E^*{}_{\tau'}\otimes_{E^{op}}F_T]\otimes A\} \\
                    &\cong \{[(A\otimes{}_{\tau'}E^*)\otimes_{E^{op}}F_W]{}_{\tau'}\otimes A\} \otimes_{A} \{[A\otimes_{\tau'}E^*{}_{\tau'}\otimes_{E^{op}}F_T]\otimes A\} \\
                     &\cong \{[(A\otimes_{\tau'}E^*)\otimes_{E^{op}}F_W]{}_{\tau'}\otimes A\} \otimes_{A} \{[(A\otimes_{\tau'}E^*)\otimes_{E^{op}}F_T]{}_{\tau'}\otimes A\} \\
                    &\cong \beta(F_W)\otimes_A \beta(F_T) \\
                    &\cong \beta(W)\otimes_A^L\beta(T)=\beta(T)\cdot \beta(W),
\end{align*}
where $(i)$ and $(k)$ are by \cite[Proposition 2.21, Proposition 2.3.6]{DNR}, $(h)$ and $(j)$ are obtained by (Isom2) and Lemma \ref{important}, respectively. Hence
$\beta$ is a group morphism.
 For any tilting DG module $X$ of $A$, let $F_X$ be a K-projective resolution of $X$. We have
\begin{align*}
\beta\circ \lambda(X) &= [(A\otimes_{\tau'}E^*)\otimes_{E^{op}}(E{}_{\tau'}\otimes X)\otimes_{\tau'}E^*]{}_{\tau'}\otimes A \\
                     &\cong A\otimes_{\tau'}E^*{}_{\tau'}\otimes X\otimes_{\tau'}E^*{}_{\tau'}\otimes A \\
                     &\cong A\otimes_{\tau'}E^*{}_{\tau'}\otimes F_X\otimes_{\tau'}E^*{}_{\tau'}\otimes A \\
                     &\cong A\otimes_{\tau'}E^*{}_{\tau'}\otimes A\otimes_A F_X\otimes_A A\otimes_{\tau'}E^*{}_{\tau'}\otimes A \\
                     &\cong A\otimes_AF_X\otimes_AA \cong F_X\cong X
\end{align*}
in  $\mathscr{D}(A^{e})$. So $\beta$ is the inverse of $\lambda$ and hence $\lambda:\mathrm{DPic}(A)^{op}\to \mathrm{DPic}(E)$ is a group isomorphism. For any group $G$, we always have a group isomorphism $\sigma: G\to G^{op}$ given by $\sigma(g)=g^{-1}$. Therefore, $\mathrm{DPic}(A)\cong \mathrm{DPic}(E)$.
\end{proof}

\begin{rem}\label{imprem}
Theorem \ref{dpd} offers a shortcut to compute the derived Picard group of a homologically smooth Koszul DG algebra $A$. Indeed, we have  $\mathrm{DPic}(A)\cong \mathrm{DPic}(E)$ by Theorem \ref{dpd}, where the Ext-algebra $E=H(R\Hom_A(k,k))$ of $A$
  is a finite dimensional local $k$-algebra. On the other hand, the derived Picard group of a finite dimensional local algebra $\mathcal{E}$ is isomorphic to $\Bbb{Z}\times \mathrm{Pic}_k (\mathcal{E})$ (cf. \cite{Zim,RZ}), and we have
$\mathrm{Pic}_k (\mathcal{E})\cong \mathrm{Out}_k(\mathcal{E})$, which is attributed to Fr\"{o}hlich \cite{Fro} and can also be found in \cite[Theorem 37.16]{Rei}.
    Therefore, the computations of the derived Picard group of a homologically smooth Koszul connected cochain DG algebra can be reduced to compute the outer automorphism group of a finite dimensional local algebra.
\end{rem}


\section{the derived picard groups of $4$ families of dg algebras }\label{gradedreg}
In this section, we will apply Theorem \ref{dpd} to compute the derived Picard group of a connected cochain DG algebra $A$ when $H(A)$ belongs to one of the following $4$
cases:
\begin{enumerate}
\item $H(A)=k[\lceil x\rceil ], x\in \mathrm{ker}(\partial_A^1)$; \\
\item $H(A)=k\langle \lceil x_1\rceil, \cdots, \lceil x_n\rceil \rangle, x_1,\cdots, x_n\in \mathrm{ker}(\partial_A^1)$;                                                 \\
\item $H(A)=\k[ \lceil x_1\rceil,\lceil x_2\rceil], x_1,x_2\in \mathrm{ker}(\partial_A^1)$; \\
\item $H(A)=\k\langle \lceil x_1\rceil,\lceil x_2\rceil\rangle/(\lceil x_1\rceil \lceil x_2\rceil+\lceil x_2\rceil \lceil x_1\rceil ), x_1,x_2\in \mathrm{ker}(\partial_A^1)$.
\end{enumerate}
For briefness, we only calculate in detail the derived Picard group of $A$ when $H(A)$ belongs to cases (2) and (3).
The computations for cases (1) and (4) are similar to the  case (3).

\begin{prop}\label{trivialdgfree}
 Let $A$ be a connected cochain DG algebra such that $$H(A)= k\langle \lceil x_1\rceil, \cdots, \lceil x_n\rceil \rangle,$$ for some degree one cocycle elements $x_1,\cdots, x_n$ in $A$.
 Then we have
 $$ \mathrm{DPic}(A)=[\Bbb{Z}\times \mathrm{GL}_n(k)].$$
 \end{prop}
\begin{proof}
By the proof of \cite[Proposition 6.2]{MXYA}, ${}_Ak$ has a minimal semi-free resolution $F$ with $$F^{\#}=A^{\#}\oplus [A^{\#}\otimes (\bigoplus\limits_{i=1}^nk\Sigma e_{x_i})]$$ and $\partial_F$ is defined by $\partial_F(\Sigma e_{x_i})=x_i, i=1,2,\cdots, n$. Since $F$ admits a semi-basis
$\{1,\Sigma e_{x_1}, \cdots, \Sigma e_{x_n}\}$, $A$ is a homologically smooth Koszul cochain DG algebra.
By the minimality of $F$, we have $$H(\Hom_A(F,k))=\Hom_A(F,k)= k\cdot 1^*\oplus [\bigoplus\limits_{i=1}^nk\cdot (\Sigma e_{x_i})^*].$$
So the Ext-algebra $E=H(\Hom_A(F,F))$ is concentrated in degree $0$. On the other hand, $$\Hom_A(F,F)^{\#}\cong \{k\cdot 1^*\oplus [\bigoplus\limits_{i=1}^nk\cdot (\Sigma e_{x_i})^*]\} \otimes_{k} F^{\#}$$ is concentrated in degree $\ge 0$. This implies that $E= Z^0(\Hom_A(F,F))$.
Since $F^{\#}$ is a free graded $A^{\#}$-module with a basis $\{1,\Sigma e_{x_1}, \cdots, \Sigma e_{x_n}\}$ concentrated in degree $0$,
  the elements in  $\Hom_A(F,F)^0$ is one to one correspondence with the matrixes in $M_{n+1}(k)$. Indeed, any $f\in \Hom_A(F,F)^0$ is uniquely determined by
  a matrix $$A_f=(a_{ij})_{(n+1)\times (n+1)}\in M_{n+1}(k)$$ with
$$\left(
                         \begin{array}{c}
                          f(1) \\
                          f(\Sigma e_{x_1})\\
                          \vdots\\
                          f(\Sigma e_{x_n})\\
                         \end{array}
                       \right) =      A_f \cdot \left(
                         \begin{array}{c}
                          1 \\
                          \Sigma e_{x_1}\\
                          \vdots\\
                          \Sigma e_{x_n}\\
                         \end{array}
                       \right).  $$
                       And $f\in  Z^0(\Hom_A(F,F)$ if and only if $\partial_{F}\circ f=f\circ \partial_{F}$, if and only if
 $$ A_f\cdot  \left(
                         \begin{array}{cccc}
                           0 & 0& \cdots & 0\\
                           x_1 & 0 & \cdots & 0\\
                           \vdots & \vdots & \ddots & \vdots\\
                          x_n & 0 & \cdots & 0 \\
                         \end{array}
                       \right)  =  \left(
                         \begin{array}{cccc}
                           0 & 0& \cdots & 0\\
                           x_1 & 0 & \cdots & 0\\
                           \vdots & \vdots & \ddots & \vdots\\
                          x_n & 0 & \cdots & 0 \\
                         \end{array}
                       \right) \cdot A_f, $$  which is also equivalent to
                       $$\begin{cases}
                       a_{ij}=0, \forall i\neq j, j=2,3,\cdots, n+1,\\
                       a_{11}=a_{22}=\cdots =a_{nn}=a_{(n+1)(n+1)}
                       \end{cases}$$
by direct computations. Hence the algebra $$ E\cong \left\{ \left(
                         \begin{array}{cccc}
                           \lambda_1 & 0& \cdots & 0\\
                           \lambda_2 & \lambda_1 & \cdots & 0\\
                           \vdots & \vdots & \ddots  & \vdots\\
                           \lambda_{n+1} & 0& \cdots & \lambda_1\\
                         \end{array}
                       \right)\quad | \quad \lambda_1,\lambda_2,\cdots,\lambda_{n+1}\in k \right\} = \mathcal{E}.$$
                       Set $$ e_1= E_{n+1}, e_2=  \left(
                         \begin{array}{cccc}
                           0 & 0& \cdots & 0\\
                           1 & 0 & \cdots & 0\\
                           \vdots & \vdots & \ddots  & \vdots\\
                           0 & 0& \cdots & 0\\
                         \end{array}
                       \right),\cdots, e_{n+1}= \left(
                         \begin{array}{cccc}
                           0 & 0& \cdots & 0\\
                           0 & 0 & \cdots & 0\\
                           \vdots & \vdots & \ddots  & \vdots\\
                           1 & 0& \cdots & 0\\
                         \end{array}
                       \right).
                       $$
                     Then $\{e_1,e_2,\cdots,e_{n+1}\}$ is a $k$-linear bases of the $k$-algebra
                        $\mathcal{E}$. The multiplication on $\mathcal{E}$ is defined by the following relations
                       $$\begin{cases} e_1\cdot e_i=e_i\cdot e_1=e_i, \forall i\in \{1,2,\cdots, n+1\}, \\
                        e_ie_j=0, \forall i,j\in \{2,3,\cdots, n+1\}.
                       \end{cases}$$
Hence $\mathcal{E}$ is a local commutative $k$-algebra isomorphic to $$\frac{k\langle y_1,y_2,\cdots, y_n\rangle}{\left(
    y_iy_j, i,j=1,2,\cdots, n \right)},|y_i|=0.$$
 Since $\{e_1,e_2,\cdots, e_{n+1}\}$ is a $k$-linear basis of $\mathcal{E}$, any $k$-linear map $\sigma: \mathcal{E}\to \mathcal{E}$ uniquely corresponds to a matrix in $C_{\sigma}=(c_{ij})_{(n+1)\times (n+1)}\in M_{n+1}(k)$
with $$\left(
                         \begin{array}{c}
                          \sigma(e_1) \\
                          \sigma(e_2)\\
                          \vdots\\
                          \sigma(e_{n+1})\\
                         \end{array}
                       \right) =      C_{\sigma} \cdot \left(
                         \begin{array}{c}
                          e_1 \\
                          e_2\\
                          \vdots\\
                          e_{n+1}\\
                         \end{array}
                       \right).  $$
                       Such $\sigma \in \mathrm{Aut}_k(\mathcal{E})$ if and only if
                      $$C_{\sigma}\in \mathrm{GL}_{n+1}(k) \quad \text{and}\quad \sigma (e_i\cdot e_j)=\sigma (e_i)\sigma (e_j), \,\,\text{for any}\,\, i,j=1,\cdots, n+1.$$
Therefore, $\sigma \in \mathrm{Aut}_k(\mathcal{E})$ if and only if
$$\begin{cases}
|(c_{ij})_{(n+1)\times (n+1)}|\neq 0,   \\
 \sigma(e_1)\sigma(e_i)=\sigma(e_i) \\
\sigma(e_i)\cdot \sigma(e_j)=0, \forall i,j\in \{2,\cdots, n+1\}
\end{cases} \Longleftrightarrow \begin{cases}
|(c_{ij})_{(n+1)\times (n+1)}|\neq 0,\\
c_{11}=1,c_{12}=\cdots=c_{1(n+1)}=0, \\
c_{21}=\cdots=c_{(n+1)1}=0
\end{cases}
$$
Then we get \begin{align*}\mathrm{Aut}_k(\mathcal{E}) &\cong  \left\{ \left(
                         \begin{array}{cccc}
                           1 & 0& \cdots& 0\\
                           0 & c_{22} & \cdots & c_{2(n+1)}\\
                           \vdots & \vdots & \ddots & \vdots\\
                           0 & c_{(n+1)2} & \cdots & c_{(n+1)(n+1)}\\
                         \end{array}
                       \right)\quad |  \left|
                         \begin{array}{ccc}
                           c_{22} & \cdots & c_{2(n+1)}\\
                            \vdots & \ddots & \vdots\\
                          c_{(n+1)2} & \cdots & c_{(n+1)(n+1)}\\
                         \end{array}
                       \right|\neq 0 \right\} \\
                       &\cong \mathrm{GL}_n(k)
                       \end{align*}
Since $\mathcal{E}$ is commutative, we have $\mathrm{Aut}_k(\mathcal{E})\cong \mathrm{Out}_k(\mathcal{E})$.
 By \cite[Proposition 3.4]{Yek1}, we have $\mathrm{Pic}_k (\mathcal{E})\cong \mathrm{Out}_k(\mathcal{E})$ and $\mathrm{DPic}_k(\mathcal{E})\cong \Bbb{Z}\times \mathrm{Pic}_k (\mathcal{E})\cong \Bbb{Z}\times \mathrm{GL}_n(k)$.
 By Theorem \ref{dpd},  we have  $$\mathrm{DPic}(A)\cong \mathrm{DPic}_k (E)\cong \mathrm{DPic}_k (\mathcal{E})\cong \Bbb{Z}\times \mathrm{GL}_n(k).$$
\end{proof}

\begin{cor}
Let $A$ be the trivial DG free algebra such that $A^{\#}=k\langle x_1,\cdots, x_n\rangle$ with $$|x_1|=|x_2|=\cdots =|x_n|=1.$$
Then  $$\mathrm{DPic}(A)\cong \Bbb{Z}\times \mathrm{GL}_n(k).$$
\end{cor}

 \begin{prop}\label{trivialpoly}
 Let $A$ be a connected cochain DG algebra with $H(A)=k[\lceil x_1\rceil, \lceil x_2\rceil]$, for some degree one cocycle elements $x_1,x_2$ in $A$.
  Then
 $$ \mathrm{DPic}(A)\cong \Bbb{Z}\times [k^2\rtimes \mathrm{SL}_2(k)\rtimes k^{\times}].$$
 \end{prop}
\begin{proof}
The trivial graded $H(A)$-module $k$ admits a minimal free resolution:
\begin{align*}
0\to H(A) \otimes ke_{12} \stackrel{d_2}{\to} H(A) \otimes (ke_{1}\oplus ke_{2}) \stackrel{d_1}{\to} H(A) \stackrel{\varepsilon_{H(A)}}{\to}  k\to 0,
\end{align*}
where $d_1$ and $d_2$ are defined by \begin{align*}
d_1(e_{1})&=\lceil x_1\rceil, \\
d_1(e_{2})&=\lceil x_2\rceil, \\
 d_2(e_{12})&=\lceil x_1 \rceil e_2- \lceil x_2\rceil e_1.
 \end{align*}
According to the constructing procedure of Eilenberg-Moore resolution, we can construct a minimal semi-free resolution of the DG $A$-module $k$ step by step.
Let $F(0)=A$ and define a morphism of DG $A$-modules $f_0:F(0)\to k$ by the augmentation map $\varepsilon_A: A\to k$. Then we extend $F(0),f_0$ to $F(1),f_1$ such that
$$F(1)^{\#}=F(0)^{\#}\oplus A^{\#}\otimes (k \Sigma e_{1}\oplus k \Sigma e_{2}),$$
$\partial_{F(1)}(\Sigma e_1)=x_1, \partial_{F(1)}(\Sigma e_2)=x_2$ and $f_1(\Sigma e_1)=f_1(\Sigma e_2)=0$.
 Let $z$  be the cocycle element in the DG free $A$-module $A\otimes (k e_{z_1}\oplus k e_{z_2})$ such that $z=x_2e_1- x_1e_2$.
 Then we have \begin{align*}
 \partial_{F(1)}(\Sigma z)&=\partial_{F(1)}[x_1\Sigma e_2- x_2\Sigma e_1] \\
                          &=-x_1x_2+x_2x_1
 \end{align*}
 Since $\lceil -x_1x_2+x_2x_1\rceil $ is zero in $H(A)$, there exist $\chi\in A^1$ such that $\partial_A(\chi)=-x_1x_2+x_2x_1$. We have
 $\partial_{F(1)}(\Sigma z -\chi)=0$.
 Define a DG $A$-module $F(2)$ such that $$F(2)^{\#}=F(1)^{\#}\oplus A^{\#}\otimes k\Sigma^2e_{12}$$ with a differential defined by $\partial_{F(2)}|_{F(1)}=\partial_{F(1)}$, and $$\partial_{F(2)}(\Sigma^2e_{12})=x_1\Sigma e_2-x_2\Sigma e_1 -\chi.$$
Since
$f_1[x_1\Sigma e_2-x_2\Sigma e_1 -\chi]=0=\partial_A(0)$,
we can extend $f_1$ to $f_2: F(2)\to k$ by $f_2(\Sigma^2e_{12})=0.$
The DG morphism $f_2:F(2)\to k$ is the Eilenberg-Moore resolution of $k$ as a DG $A$-module.
Set $F=F(2)$.
By the constructing procedure above, we have $\partial_{F}(F)\subseteq A^{\ge 1}F$ and $F$ admits a semi-basis
$$\{1,\Sigma e_1, \Sigma e_1, \Sigma^2e_{12}\}$$ satisfying $\partial_F(1)=0;  \partial_F(\Sigma e_1)=x_1, \partial_F(\Sigma e_2)=x_2;$ and $\partial_F(\Sigma^2e_{12})=x_1\Sigma e_2-x_2\Sigma e_1 -\chi$. Hence $A$ is homologically smooth and Koszul.
By the minimality of $F$, we have $$H(\Hom_A(F,k))=\Hom_A(F,k)= k\cdot 1^*\oplus k(\Sigma e_1)^*\oplus k(\Sigma e_2)^*\oplus k(\Sigma^2 e_{12})^*.$$
  So the Ext-algebra $E=H(\Hom_A(F,F))$  is concentrated in degree $0$. On the other hand, $$\Hom_A(F,F)^{\#}\cong \{k\cdot 1^*\oplus k(\Sigma e_1)^*\oplus k(\Sigma e_2)^*\oplus k(\Sigma^2 e_{12})^*\}\otimes_{k} F^{\#}$$ is concentrated in degree $\ge 0$. This implies that $E= Z^0(\Hom_A(F,F))$.
Since $F^{\#}$ is a free graded $A^{\#}$-module with a basis
$$\{1,\Sigma e_1, \Sigma e_1, \Sigma^2e_{12}\},$$ the elements in  $\Hom_A(F,F)^0$ is in one to one correspondence with the matrixes in $M_{4}(k)$. Indeed, any $f\in \Hom_A(F,F)^0$ is uniquely determined by
  a matrix $$A_f=(a_{ij})_{4\times 4}\in M_{4}(k).$$
  We have
  $$
  \left(
                         \begin{array}{c}
                          f(1) \\
                          f(\Sigma e_1)\\
                          f(\Sigma e_2)\\
                          f(\Sigma^2 e_{12})\\
                         \end{array}
                       \right) =      A_f \cdot \left(
                         \begin{array}{c}
                          1 \\
                          \Sigma e_1\\
                          \Sigma e_2\\
                          \Sigma^2 e_{12}\\
                         \end{array}
                       \right).
                       $$
 And $f\in  Z^0(\Hom_A(F,F)$ if and only if $\partial_{F}\circ f=f\circ \partial_{F}$, which is equivalent to
 $$ A_f \left(
                         \begin{array}{cccc}
                           0 & 0& 0 & 0\\
                           x_1 & 0 & 0 & 0\\
                           x_2 & 0 & 0 & 0\\
                          -\chi & -x_2 & x_1 & 0 \\
                         \end{array}
                       \right) =  \left(
                         \begin{array}{cccc}
                           0 & 0& 0 & 0\\
                           x_1 & 0 & 0 & 0\\
                           x_2 & 0 & 0 & 0\\
                          -\chi & -x_2 & x_1 & 0 \\
                         \end{array}
                       \right) A_f.$$
By computations, one has
$$
\begin{cases}
a_{ij}=0, \forall i<j,\\
a_{11}=a_{22}=a_{33}=a_{44}, \\
a_{21}=-a_{43}, a_{31}=a_{42}. \\
\end{cases}
$$
Hence the algebra $$ E\cong \{ \left(
                         \begin{array}{cccc}
                           \lambda_1 & 0& 0 & 0\\
                           \lambda_2 & \lambda_1 & 0 & 0\\
                           \lambda_3 & 0 & \lambda_1  & 0\\
                           \lambda_4 & \lambda_3 & -\lambda_2 & \lambda_1\\
                         \end{array}
                       \right)\quad | \quad \lambda_1,\lambda_2,\lambda_3,\lambda_4\in k \}=\mathcal{E}. $$
Set $$ e_1= E_{4}, e_2=  \left(
                         \begin{array}{cccc}
                           0 & 0 & 0 & 0\\
                           1 & 0 & 0 & 0\\
                           0 & 0 & 0 & 0\\
                           0 & 0 & -1 & 0\\
                         \end{array}
                       \right), e_{3}= \left(
                         \begin{array}{cccc}
                           0 & 0 & 0 & 0\\
                           0 & 0 & 0 & 0\\
                           1 & 0 & 0 & 0\\
                           0 & 1 & 0 & 0\\
                         \end{array}
                       \right), e_{4}= \left(
                         \begin{array}{cccc}
                           0 & 0 & 0 & 0\\
                           0 & 0 & 0 & 0\\
                           0 & 0 & 0 & 0\\
                           1 & 0 & 0 & 0\\
                         \end{array}
                       \right).
                       $$
                     Then $\{e_1,e_2,e_3,e_{4}\}$ are $k$-linear bases of the $k$-algebra $\mathcal{E}$. We have
                     \begin{align*}
                     \begin{cases}
                     e_1e_i=e_i, i=1,2,3,4\\
                     e_2e_3=-e_3e_2=-e_4 \\
                     (e_2)^2=(e_3)^2=(e_4)^2=0\\
                     e_2e_4=e_4e_2=e_3e_4=e_4e_3=0.
                     \end{cases}
                     \end{align*}
Any $k$-linear map $\sigma: \mathcal{E}\to \mathcal{E}$ is uniquely corresponding to a matrix in $C_{\sigma}=(c_{ij})_{4\times 4}\in M_{4}(k)$,
with $$\left(
                         \begin{array}{c}
                          \sigma(e_1) \\
                          \sigma(e_2)\\
                          \sigma(e_3)\\
                          \sigma(e_4)\\
                         \end{array}
                       \right) =   C_{\sigma}  \left(
                         \begin{array}{c}
                          e_1 \\
                          e_2\\
                          e_3\\
                          e_4\\
                         \end{array}
                       \right).$$
We have $$\sigma \in \mathrm{Aut}_k(\mathcal{E}) \Leftrightarrow
                      C_{\sigma}\in \mathrm{GL}_{4}(k), \sigma (e_i\cdot e_j)=\sigma (e_i)\sigma (e_j), \forall  i,j=1,\cdots, 4.$$
By computations,
\begin{align*}
&\sigma \in \mathrm{Aut}_k(\mathcal{E})  \Leftrightarrow     \begin{cases}
c_{11}=1,c_{1j}=c_{j1}=0,j=2,3,4 \\
c_{42}=c_{43}=0, c_{44}=c_{22}c_{33}-c_{23}c_{32}.
\end{cases}
\end{align*}
Then we get \begin{align*}\mathrm{Aut}_k(\mathcal{E}) &\cong  \left\{ \left(
                         \begin{array}{cccc}
                           1 & 0& 0& 0\\
                           0 & a & d & e\\
                           0 & c & b & f\\
                           0 & 0 & 0 & ab-cd\\
                         \end{array}
                       \right)\,\, |  a,b,c,d,e,f\in k, ab-cd\neq 0\right \}\\
                       &\cong \left\{ \left(
                         \begin{array}{ccc}
                           a & d & e\\
                           c & b & f\\
                           0 & 0 & ab-cd\\
                         \end{array}
                       \right)\,\, |  a,b,c,d,e,f\in k, ab-cd\neq 0\right \}\\
                       &\cong \left\{ \left(
                         \begin{array}{ccc}
                           a & d & e\\
                           c & b & f\\
                           0 & 0 & 1\\
                         \end{array}
                       \right)\,\, |  a,b,c,d,e,f\in k, ab-cd= 1\right \} \rtimes k^{\times}\\
                       &\cong k^2\rtimes \mathrm{SL}_2(k)\rtimes k^{\times}.
                       \end{align*}
Since $e_1$ is the unique invertible element of $\mathcal{E}$, we have
$\mathrm{Aut}_k(\mathcal{E})\cong \mathrm{Out}_k(\mathcal{E})$.
By Remark \ref{imprem}, we have $\mathrm{Pic}_k (\mathcal{E})\cong \mathrm{Out}_k(\mathcal{E})$ and $$\mathrm{DPic}_k(\mathcal{E})\cong \Bbb{Z}\times \mathrm{Pic}_k (\mathcal{E})\cong \Bbb{Z}\times [k^2\rtimes \mathrm{SL}_2(k)\rtimes k^{\times}].$$
 By Theorem \ref{dpd},  we have $$\mathrm{DPic}(A)\cong \mathrm{DPic}_k (E)\cong \mathrm{DPic}_k (\mathcal{E})\cong \Bbb{Z}\times [k^2\rtimes \mathrm{SL}_2(k)\rtimes k^{\times}].$$
\end{proof}
\begin{cor}
Let $A$ be the trivial DG free algebra such that $A^{\#}=k[x_1, x_2]$ with $|x_1|=|x_2|=1.$
Then  $$\mathrm{DPic}(A)\cong \Bbb{Z}\times [k^2\rtimes \mathrm{SL}_2(k)\rtimes k^{\times}].$$
\end{cor}
By a similar proof, we can get the following two proposition.
\begin{prop}
Let $(A,\partial_A)$ be a connected cochain DG algebra such that $$H(A)= k\langle \lceil x\rceil,\lceil y\rceil \rangle/(\lceil x\rceil \lceil y\rceil+\lceil y\rceil \lceil x\rceil),$$ for some degree one cocycle elements $x,y$ in $A$.
Then $$ \mathrm{DPic}(A)\cong \Bbb{Z}\times  \left\{ \left(
                         \begin{array}{ccc}
                            a & d & e\\
                            c & b & f\\
                            0 & 0 & ab+cd\\
                         \end{array}
                       \right)\,\, |  a,b,c,d,e,f\in k, a^2b^2\neq c^2d^2, ad=cb=0 \right\} .$$
\end{prop}

\begin{prop}\label{polyone}
Let $A$ be the connected cochain DG algebra such that $H(A)=k[\lceil x\rceil]$, where $x$ is a cocycle element in $A^1$. Then  $$\mathrm{DPic}(A)\cong \Bbb{Z}\times k^{\times} .$$
\end{prop}

\begin{cor}\label{impcase}
Let $A$ be the DG free algebra such that $A^{\#}=\k\langle x_1,x_2\rangle, |x_1|=|x_2|=1$, and  $\partial_{A}$ is defined by
$\partial_{A}(x_1)=x_1^2$ and $\partial_{A}(x_2)=0.$ Then we have $$\mathrm{DPic}(A)\cong \Bbb{Z}\times k^{\times}.$$
\end{cor}
\begin{proof}
By \cite[Proposition 5.1(5)]{MXYA}, $H(A)=k[\lceil x_2\rceil]$. So $A$ is a Koszul Calabi-Yau DG algebra, and $$\mathrm{DPic}(A)\cong \Bbb{Z}\times k^{\times}$$ by Proposition \ref{polyone}.
\end{proof}

\begin{cor}\label{impcase}
Let $A$ be the DG free algebra such that $A^{\#}=\k\langle x_1,x_2\rangle, |x_1|=|x_2|=1$, and  $\partial_{A}$ is defined by
$\partial_{A}(x_1)=x_2^2$ and $\partial_{A}(x_2)=x_2^2.$ Then we have $$\mathrm{DPic}(A)\cong \Bbb{Z}\times k^{\times}.$$
\end{cor}
\begin{proof}
By \cite[Proposition 5.1(8)]{MXYA}, $H(A)=k[\lceil x_1 - x_2\rceil]$. So $A$ is a Koszul Calabi-Yau DG algebra, and $$\mathrm{DPic}(A)\cong \Bbb{Z}\times k^{\times}$$ by Proposition \ref{polyone}.
\end{proof}

\section{applications to some other special cases}\label{special}
By \cite[Example 2.11]{MW2}, \cite[Proposition 6.5]{MXYA} and \cite[Proposition 6.1]{MHLX}, one sees that there are homologically smooth Koszul connected cochain DG algebras, whose cohomology graded algebras are not Koszul and regular. For those special cases, we will apply Theorem \ref{dpd} to compute their derived Picard groups in this section.

\begin{exa}\label{dgfree}\cite{MXYA}
Let $A$ be the DG free algebra such that $A^{\#}=\k\langle x_1,x_2\rangle, |x_1|=|x_2|=1$, and  $\partial_{A}$ is defined by
$\partial_{A}(x_1)=x_2^2$ and $\partial_{A}(x_2)=0.$ We have $H(A)=k[\lceil x_2\rceil,\lceil x_1x_2+x_2x_1\rceil]/(\lceil x_2\rceil^2)$
by \cite[Proposition 5.1(6)]{MXYA}. One sees that $H(A)$ is not Koszul and $\mathrm{gl.dim}H(A)=\infty$. In spite of this, $A$ is a Koszul Calabi-Yau DG algebra
by \cite[Proposition 6.3]{MXYA}. We can still apply Theorem \ref{dpd} to compute its derived Picard group.
\end{exa}

\begin{prop}\label{impcase}
Let $A$ be the DG free algebra in Example \ref{dgfree}. Then we have $$\mathrm{DPic}(A)\cong  [k \rtimes \mathrm{SL}_1(k)]\rtimes k^{\times}.
.$$
\end{prop}
\begin{proof}
 From the proof of \cite[Proposition 6.3]{MXYA},  one sees that the trivial DG $A$-module $k$ admits a minimal semi-free resolution $f:F\to k$ with $$F^{\#}=A^{\#}\oplus A^{\#}\Sigma e_{x_2}\oplus A^{\#}\Sigma e_z$$ and a differential $\partial_F$ defined by $\partial_F(\Sigma e_{x_2})=x_2$ and $\partial_F(\Sigma e_z)=x_1+x_2\Sigma e_{x_2}$. By the minimality of $F$, we have $$H(\Hom_A(F,k))=\Hom_A(F,k)= k\cdot 1^*\oplus k\cdot(\Sigma e_{x_2})^*\oplus k \cdot(\Sigma e_z)^*.$$
  So the Ext-algebra $E=H(\Hom_A(F,F))$  is concentrated in degree $0$. On the other hand, $$\Hom_A(F,F)^{\#}\cong [k \cdot 1^*\oplus k \cdot (\Sigma e_{x_2})^*\oplus k \cdot (\Sigma e_z)^*]\otimes_{k} F^{\#}$$ is concentrated in degree $\ge 0$. This implies that $E= Z^0(\Hom_A(F,F))$.
Since $F^{\#}$ is a free graded $A^{\#}$-module with a basis
$\{1,\Sigma e_{x_2},\Sigma e_z\}$ concentrated in degree $0$,
  the elements in  $\Hom_A(F,F)^0$ is in one to one correspondence with the matrixes in $M_3(k)$. Indeed, any $f\in \Hom_A(F,F)^0$ is uniquely determined by
  a matrix $A_f=(a_{ij})_{3\times 3}\in M_3(\k)$ with
$$\left(
                         \begin{array}{c}
                          f(1) \\
                          f(\Sigma e_{x_2})\\
                          f(\Sigma e_z)\\
                         \end{array}
                       \right) =      A_f \cdot \left(
                         \begin{array}{c}
                          1 \\
                          \Sigma e_{x_2}\\
                          \Sigma e_z\\
                         \end{array}
                       \right).  $$
                       And $f\in  Z^0(\Hom_A(F,F))$ if and only if $\partial_{F}\circ f=f\circ \partial_{F}$, if and only if
 $$ A_f\cdot \left(
                         \begin{array}{ccc}
                           0 & 0& 0 \\
                           x_2 & 0 & 0 \\
                           x_1 & x_2 & 0 \\
                         \end{array}
                       \right) =  \left(
                         \begin{array}{ccc}
                           0 & 0& 0 \\
                           x_2 & 0 & 0\\
                           x_1 & x_2 & 0 \\
                         \end{array}
                       \right) \cdot A_f, $$  which is also equivalent to
                       $$\begin{cases}
                       a_{12}=a_{13}=a_{23}=0\\
                       a_{11}=a_{22}=a_{33}\\
                       a_{21}=a_{32}
                       \end{cases}$$
by direct computations.
Hence  Hence the algebra $$ E\cong \left\{ \left(
                         \begin{array}{ccc}
                           a & 0& 0\\
                           b & a & 0 \\
                           c & b & a \\
                         \end{array}
                       \right)\quad | \quad a,b,c\in k \right\} = \mathcal{E}.$$
                       Set \begin{align*} e_1= \left(
                         \begin{array}{ccc}
                           1 & 0& 0\\
                           0 & 1 & 0 \\
                           0 & 0 & 1 \\
                         \end{array}
                       \right),& e_2= \left(
                         \begin{array}{ccc}
                           0 & 0& 0\\
                           1 & 0 & 0 \\
                           0 & 1 & 0 \\
                         \end{array}
                       \right),
                        e_3= \left(
                         \begin{array}{ccc}
                           0 & 0& 0\\
                           0 & 0 & 0 \\
                           1 & 0 & 0 \\
                         \end{array}
                       \right).
                       \end{align*}
 Then $\{e_1,e_2,e_3\}$ is a $k$-linear bases of the $k$-algebra
                        $\mathcal{E}$. The multiplication on $\mathcal{E}$ is defined by the following relations
                       $$\begin{cases} e_1\cdot e_i=e_i\cdot e_1=e_i, i=1,2,3 \\
                        e_2^2=e_3, e_2\cdot e_3=e_3\cdot e_2 =0, \\
                       e_3^2=0
                       \end{cases} .$$
Hence $\mathcal{E}$ is a local commutative $k$-algebra isomorphic to $k[x]/(x^3)$.
 Since $\{e_1,e_2,e_3\}$ is a $k$-linear basis of $\mathcal{E}$, any $k$-linear map $\sigma: \mathcal{E}\to \mathcal{E}$ uniquely corresponds to a matrix in $C_{\sigma}=(c_{ij})_{4\times 4}\in M_3(\k)$
with $$\left(
                         \begin{array}{c}
                          \sigma(e_1) \\
                          \sigma(e_2)\\
                          \sigma(e_3)\\
                         \end{array}
                       \right) =      C_{\sigma} \cdot \left(
                         \begin{array}{c}
                          e_1 \\
                          e_2\\
                          e_3\\
                         \end{array}
                       \right).  $$
                       Such $\sigma \in \mathrm{Aut}_k(\mathcal{E})$ if and only if
                      $$C_{\sigma}\in \mathrm{GL}_3(\k) \quad \text{and}\quad \sigma (e_i\cdot e_j)=\sigma (e_i)\sigma (e_j), \,\,\text{for any}\,\, i,j=1,2,3.$$
Therefore, $\sigma \in \mathrm{Aut}_k(\mathcal{E})$ if and only if
$$\begin{cases}
|(c_{ij})_{3\times 3}|\neq 0, \sigma(e_1)=e_1  \\
 [\sigma(e_2)]^2=\sigma(e_3), [\sigma(e_3)]^2=0\\
\sigma(e_2)\cdot \sigma(e_3)=\sigma(e_3)\cdot\sigma(e_2)=0
\end{cases} \Longleftrightarrow \begin{cases}
c_{22}\neq 0, c_{11}=1, c_{12}=c_{13}=0 \\
c_{21}=c_{31}=c_{32}=0\\
c_{33}=c_{22}^2.
\end{cases}
$$
Then we get \begin{align*}\mathrm{Aut}_k(\mathcal{E}) &\cong  \left \{ \left(
                         \begin{array}{ccc}
                           1 & 0& 0\\
                           0 & a & b \\
                           0 & 0 & a^2\\
                         \end{array}
                       \right)\quad | \quad a\in k^{\times}, b\in \k \right\}\\
                       &\cong \left \{ \left(
                         \begin{array}{cc}
                            a & b \\
                            0 & a^2\\
                         \end{array}
                       \right)\quad | \quad a\in k^{\times}, b\in \k \right\}\\
                       &\cong \left \{ \left(
                         \begin{array}{cc}
                            1 & b \\
                            0 & 1\\
                         \end{array}
                       \right)\quad | \quad b\in \k \right\} \rtimes k^{\times}\\
                       &\cong [k \rtimes \mathrm{SL}_1(k)]\rtimes k^{\times}.
                       \end{align*}
Since $\mathcal{E}$ is commutative, we have $\mathrm{Aut}_k(\mathcal{E})\cong \mathrm{Out}_k(\mathcal{E})$.
 By Remark \ref{imprem}, we have $\mathrm{Pic}_k (\mathcal{E})\cong \mathrm{Out}_k(\mathcal{E})$ and $\mathrm{DPic}_k(\mathcal{E})=\Bbb{Z}\times \mathrm{Pic}_k (\mathcal{E})$. So
$$\mathrm{Pic}_k (E)\cong \mathrm{Pic}_k (\mathcal{E})\cong  [k \rtimes \mathrm{SL}_1(k)]\rtimes k^{\times}$$
and  Theorem \ref{dpd} implies   $$\mathrm{DPic}(A)\cong \mathrm{DPic}_k (E)\cong \Bbb{Z}\times [k \rtimes \mathrm{SL}_1(k)]\rtimes k^{\times}.$$

\end{proof}

\begin{cor}
Let $A$ be the DG free algebra such that $\mathcal{A}^{\#}=\k\langle x_1,x_2\rangle, |x_1|=|x_2|=1$, and the differential $\partial_{A}$ is defined by
$$\partial_{A}(x_1)=-x_1^2+x_1x_2+x_2x_1-x_2^2=\partial_{A}(x_2).$$ Then we have $$\mathrm{DPic}(A)\cong \Bbb{Z}\times [k \rtimes \mathrm{SL}_1(k)]\rtimes k^{\times}.$$

\end{cor}
\begin{proof}
Let $A'$ be the connected DG algebra such that $A'^{\#}=k\langle x,y\rangle$ and its differential is defined by $\partial_{A'}(x)=y^2$ and $\partial_{A'}(y)=0$. By Proposition \ref{impcase}, we have $$\mathrm{DPic}(A')\cong  \Bbb{Z}\times [k \rtimes \mathrm{SL}_1(k)]\rtimes k^{\times}.$$
On the other hand,  $A\cong A'$ by \cite[Proposition 4.5(2)]{MXYA}. So  $$\mathrm{DPic}(A)\cong \Bbb{Z}\times [k \rtimes \mathrm{SL}_1(k)]\rtimes k^{\times}.$$

\end{proof}

\begin{exa}\label{exmw}\cite[Example2.11 ]{MW2}
Let $(A,\partial_A)$ be a connected cochain DG algebra such that $A^{\#} = k\langle x,y\rangle/(xy+yx),  |x| = |y| =1,$
with $\partial_A(x) = y^2, \partial_A(y) = 0$. By \cite{MW2}, we have $H(A)=k[\lceil x\rceil^2,\lceil y\rceil]/(\lceil y\rceil^2)$. One sees that $H(A)$ is not Koszul and $\mathrm{gl.dim}H(A)=\infty$. In spite of this, $A$ is a Koszul Calabi-Yau DG algebra
by \cite[Example 7.1]{HM}. We can still apply Theorem \ref{dpd} to compute its derived Picard group.

\end{exa}

\begin{prop} Let $A$ the connected cochain DG algebra in Example \ref{exmw}. Then $$\mathrm{DPic}_kA\cong \Bbb{Z}\times  \left\{ \left(
                         \begin{array}{ccc}
                           a & b & c\\
                           0 & a^2 & 2ab\\
                           0 & 0 & a^3\\
                         \end{array}
                       \right)\quad | \quad a\in \k^{\times}, b,c\in \k \right\}.$$
\end{prop}

\begin{proof}
By \cite[Section 2]{MW2}, the trivial DG module $k$ over the DG algebra $A$ admits a minimal semi-free resolution $F$ such that $$F^{\#}=A^{\#}\oplus A^{\#}\Sigma e_y\oplus A^{\#}\Sigma e_z\oplus A^{\#}\Sigma e_t$$ with $\partial_{F} (\Sigma e_y)=y, \partial_{F_k}(\Sigma e_z)=x+y\Sigma e_y$ and $\partial_{F}(\Sigma e_t)=x\Sigma e_y + y\Sigma e_z$. Since $F$ has a semi-basis $\{1,\Sigma e_y,\Sigma e_z, \Sigma e_t\}$ concentrated in degree $0$,  $A$ is a Koszul and homologically smooth DG algebra.
By the minimality of $F$, we have $$H(\Hom_A(F,k))=\Hom_A(F,k)= k\cdot 1^*\oplus k\cdot(\Sigma e_y)^*\oplus k \cdot(\Sigma e_z)^*\oplus k \cdot (\Sigma e_t)^*.$$
  So the Ext-algebra $E=H(\Hom_A(F,F))$  is concentrated in degree $0$. On the other hand, $$\Hom_A(F,F)^{\#}\cong (k \cdot 1^*\oplus k \cdot (\Sigma e_y)^*\oplus k \cdot (\Sigma e_z)^*\oplus k \cdot (\Sigma e_t)^*)\otimes_{k} F^{\#}$$ is concentrated in degree $\ge 0$. This implies that $E= Z^0(\Hom_A(F,F))$.
Since $F^{\#}$ is a free graded $A^{\#}$-module with a basis
$\{1,\Sigma e_y,\Sigma e_z, \Sigma e_t\}$ concentrated in degree $0$,
  the elements in  $\Hom_A(F,F)^0$ is in one to one correspondence with the matrixes in $M_n(k)$. Indeed, any $f\in \Hom_A(F,F)^0$ is uniquely determined by
  a matrix $A_f=(a_{ij})_{4\times 4}\in M_4(\k)$ with
$$\left(
                         \begin{array}{c}
                          f(1) \\
                          f(\Sigma e_y)\\
                          f(\Sigma e_z)\\
                          f(\Sigma e_t)\\
                         \end{array}
                       \right) =      A_f \cdot \left(
                         \begin{array}{c}
                          1 \\
                          \Sigma e_y\\
                          \Sigma e_z\\
                          \Sigma e_t\\
                         \end{array}
                       \right).  $$
                       And $f\in  Z^0(\Hom_A(F,F)$ if and only if $\partial_{F}\circ f=f\circ \partial_{F}$, if and only if
 $$ A_f\cdot \left(
                         \begin{array}{cccc}
                           0 & 0& 0 & 0\\
                           y & 0 & 0 & 0\\
                           x & y & 0 & 0\\
                          0 & x & y & 0 \\
                         \end{array}
                       \right) =  \left(
                         \begin{array}{cccc}
                           0 & 0& 0 & 0\\
                           y & 0 & 0 & 0\\
                           x & y & 0 & 0\\
                          0 & x & y & 0 \\
                         \end{array}
                       \right) \cdot A_f, $$  which is also equivalent to
                       $$\begin{cases}
                       a_{12}=a_{13}=a_{14}=a_{23}=a_{24}=a_{34}=0\\
                       a_{11}=a_{22}=a_{33}=a_{44}\\
                       a_{21}=a_{32}=a_{43}
                       \end{cases}$$
by direct computations. Hence the algebra $$ E\cong \left\{ \left(
                         \begin{array}{cccc}
                           a & 0& 0& 0\\
                           b & a & 0 & 0\\
                           c & b & a & 0\\
                           d & c & b & a\\
                         \end{array}
                       \right)\quad | \quad a,b,c,d\in k \right \} = \mathcal{E}.$$
                       Set \begin{align*} e_1= \left(
                         \begin{array}{cccc}
                           1 & 0& 0& 0\\
                           0 & 1 & 0 & 0\\
                           0 & 0 & 1 & 0\\
                           0 & 0 & 0 & 1\\
                         \end{array}
                       \right),& e_2= \left(
                         \begin{array}{cccc}
                           0 & 0& 0& 0\\
                           1 & 0 & 0 & 0\\
                           0 & 1 & 0 & 0\\
                           0 & 0 & 1 & 0\\
                         \end{array}
                       \right),\\
                        e_3= \left(
                         \begin{array}{cccc}
                           0 & 0& 0& 0\\
                           0 & 0 & 0 & 0\\
                           1 & 0 & 0 & 0\\
                           0 & 1 & 0 & 0\\
                         \end{array}
                       \right), & e_4= \left(
                         \begin{array}{cccc}
                           0 & 0& 0& 0\\
                           0 & 0 & 0 & 0\\
                           0 & 0 & 0 & 0\\
                           1 & 0 & 0 & 0\\
                         \end{array}
                       \right).
                       \end{align*}
                     Then $\{e_1,e_2,e_3,e_4\}$ is a $k$-linear bases of the $k$-algebra
                        $\mathcal{E}$. The multiplication on $\mathcal{E}$ is defined by the following relations
                       $$\begin{cases} e_1\cdot e_i=e_i\cdot e_1=e_i, i=1,2,3,4 \\
                        e_2^2=e_3, e_2\cdot e_3=e_3\cdot e_2 =e_4, e_2\cdot e_4=e_4\cdot e_2=0  \\
                       e_3^2=e_3\cdot e_4=e_4^2=0
                       \end{cases} .$$
Hence $\mathcal{E}$ is a local commutative $k$-algebra isomorphic to $k[x]/(x^4)$.
 Since $\{e_1,e_2,e_3,e_4\}$ is a $k$-linear basis of $\mathcal{E}$, any $k$-linear map $\sigma: \mathcal{E}\to \mathcal{E}$ uniquely corresponds to a matrix in $C_{\sigma}=(c_{ij})_{4\times 4}\in M_4(\k)$
with $$\left(
                         \begin{array}{c}
                          \sigma(e_1) \\
                          \sigma(e_2)\\
                          \sigma(e_3)\\
                          \sigma(e_4)\\
                         \end{array}
                       \right) =      C_{\sigma} \cdot \left(
                         \begin{array}{c}
                          e_1 \\
                          e_2\\
                          e_3\\
                          e_4\\
                         \end{array}
                       \right).  $$
                       Such $\sigma \in \mathrm{Aut}_k(\mathcal{E})$ if and only if
                      $$C_{\sigma}\in \mathrm{GL}_4(\k) \quad \text{and}\quad \sigma (e_i\cdot e_j)=\sigma (e_i)\sigma (e_j), \,\,\text{for any}\,\, i,j=1,\cdots, 4.$$
Therefore, $\sigma \in \mathrm{Aut}_k(\mathcal{E})$ if and only if
$$\begin{cases}
|(c_{ij})_{4\times 4}|\neq 0, \sigma(e_1)=e_1  \\
 [\sigma(e_2)]^2=\sigma(e_3), [\sigma(e_3)]^2=[\sigma(e_4)]^2=0\\
\sigma(e_2)\cdot \sigma(e_3)=\sigma(e_3)\cdot\sigma(e_2)=\sigma(e_4) \\
\sigma(e_3)\cdot \sigma(e_4)=\sigma(e_4)\cdot \sigma(e_3)=0
\end{cases} \Longleftrightarrow \begin{cases}
c_{22}\neq 0, c_{11}=1, c_{12}=c_{13}=c_{14}=0 \\
c_{21}=c_{31}=c_{32}=c_{41}=c_{42}=c_{43}=0\\
c_{33}=c_{22}^2, c_{44}=c_{22}^3, c_{34}=2c_{22}c_{23}
\end{cases}
$$
Then we get \begin{align*}
\mathrm{Aut}_k(\mathcal{E})& \cong  \left\{ \left(
                         \begin{array}{cccc}
                           1 & 0& 0& 0\\
                           0 & a & b & c\\
                           0 & 0 & a^2 & 2ab\\
                           0 & 0 & 0 & a^3\\
                         \end{array}
                       \right)\quad | \quad a\in k^{\times}, b,c\in \k \right\}\\
                       &\cong \left\{ \left(
                         \begin{array}{ccc}
                           a & b & c\\
                           0 & a^2 & 2ab\\
                           0 & 0 & a^3\\
                         \end{array}
                       \right)\quad | \quad a\in k^{\times}, b,c\in \k \right\}.
                       &\cong
\end{align*}
Since $\mathcal{E}$ is commutative, we have $\mathrm{Aut}_k(\mathcal{E})\cong \mathrm{Out}_k(\mathcal{E})$.
 By Remark \ref{imprem}, we have $\mathrm{Pic}_k (\mathcal{E})\cong \mathrm{Out}_k(\mathcal{E})$ and $\mathrm{DPic}_k(\mathcal{E})=\Bbb{Z}\times \mathrm{Pic}_k (\mathcal{E})$.
Thus
$$\mathrm{Pic}_k (E)\cong \mathrm{Pic}_k (\mathcal{E})\cong  \left\{ \left(
                         \begin{array}{ccc}
                            a & b  & c\\
                            0 & a^2 & 2ab\\
                            0 & 0 & a^3\\
                         \end{array}
                       \right)\quad | \quad a\in \k^{\times}, b,c\in \k \right\},$$
 and Theorem \ref{dpd} implies  $$\mathrm{DPic}(A)\cong \mathrm{DPic}_k (E)^{op}\cong \Bbb{Z}\times \left\{ \left(
                         \begin{array}{ccc}
                            a & b & c\\
                            0 & a^2 & 2ab\\
                            0 & 0 & a^3\\
                         \end{array}
                       \right)\quad | \quad a\in k^{\times}, b,c\in k \right\}.$$

\end{proof}

\begin{exa}\label{dgfreesp}\cite{MH}
Let $A$ be the DG free algebra such that $A^{\#}=k\langle x,y\rangle/(xy+yx)$, $|x|=|y|=1$ and $\partial_{\mathcal{A}}$ is defined by
$\partial_{A}(x)=x^2+y^2=\partial_{A}(y)$.  By \cite[Proposition 3.3(7)]{MH}, we have $H(A)=k[\lceil y-x \rceil,\lceil y^2\rceil]/(\lceil y-x\rceil^2)$. One sees that $H(A)$ is neither Koszul nor regular. In spite of this $A$ is a Koszul Calabi-Yau DG algebra by \cite[Proposition 4.3]{MH}. We can apply Theorem \ref{dpd} to compute its derived Picard group.
\end{exa}

\begin{prop}
Let $A$ be the connected cochain DG algebra in Example \ref{dgfreesp}.
Then $$\mathrm{DPic}(A)\cong \Bbb{Z}\times [k \rtimes \mathrm{SL}_1(k)]\rtimes k^{\times}.$$
\end{prop}
\begin{proof}
By \cite[Proposition 4.3]{MH}, $A$ is a Koszul Calabi-Yau DG algebra. From the proof of \cite[(A2)]{MH}, one sees that ${}_Ak$ admits a minimal semi-free resolution $F$ with $$F^{\#}=A^{\#}\oplus A^{\#}e_z\oplus A^{\#}e_t$$ and a differential $\partial_F$ defined by $\partial_F(e_z)=x-y, \partial_F(e_t)=x+(x-y)e_z$.
By the minimality of $F$, we have $$H(\Hom_A(F,k))=\Hom_A(F,k)= k\cdot 1^*\oplus k\cdot(e_z)^*\oplus k \cdot(e_t)^*.$$
  So the Ext-algebra $E=H(\Hom_A(F,F))$  is concentrated in degree $0$. On the other hand, $$\Hom_A(F,F)^{\#}\cong (k \cdot 1^*\oplus k \cdot (e_z)^*\oplus k \cdot (e_t)^*)\otimes_{k} F^{\#}$$ is concentrated in degree $\ge 0$. This implies that $E= Z^0(\Hom_A(F,F))$.
Since $F^{\#}$ is a free graded $A^{\#}$-module with a basis
$\{1, e_z, e_t\}$ concentrated in degree $0$,
  the elements in  $\Hom_A(F,F)^0$ is one to one correspondence with the matrixes in $M_3(k)$. Indeed, any $f\in \Hom_A(F,F)^0$ is uniquely determined by
  a matrix $A_f=(a_{ij})_{3\times 3}\in M_3(\k)$ with
$$\left(
                         \begin{array}{c}
                          f(1) \\
                          f(e_z)\\
                          f(e_t)\\
                         \end{array}
                       \right) =      A_f \cdot \left(
                         \begin{array}{c}
                          1 \\
                           e_z\\
                           e_t\\
                         \end{array}
                       \right).  $$
                       And $f\in  Z^0(\Hom_A(F,F)$ if and only if $\partial_{F}\circ f=f\circ \partial_{F}$, if and only if
 $$ A_f\cdot \left(
                         \begin{array}{ccc}
                           0 & 0& 0\\
                           x-y & 0 & 0\\
                           x & x-y & 0 \\
                         \end{array}
                       \right) =  \left(
                         \begin{array}{cccc}
                             0 & 0 & 0\\
                           x-y & 0 & 0\\
                           x & x-y & 0 \\
                         \end{array}
                       \right) \cdot A_f, $$  which is also equivalent to
                       $$\begin{cases}
                       a_{12}=a_{13}=a_{23}=0\\
                       a_{11}=a_{22}=a_{33}\\
                       a_{21}=a_{32}
                       \end{cases}$$
by direct computations. Hence the algebra $$ E\cong \left\{ \left(
                         \begin{array}{ccc}
                           a & 0& 0\\
                           b & a & 0 \\
                           c & b & a \\
                         \end{array}
                       \right)\quad | \quad a,b,c\in k \right\} = \mathcal{E}.$$
                       Set \begin{align*} e_1= \left(
                         \begin{array}{ccc}
                           1 & 0& 0\\
                           0 & 1 & 0 \\
                           0 & 0 & 1 \\
                         \end{array}
                       \right),& e_2= \left(
                         \begin{array}{ccc}
                           0 & 0& 0\\
                           1 & 0 & 0 \\
                           0 & 1 & 0 \\
                         \end{array}
                       \right),
                        e_3= \left(
                         \begin{array}{ccc}
                           0 & 0& 0\\
                           0 & 0 & 0 \\
                           1 & 0 & 0 \\
                         \end{array}
                       \right).
                       \end{align*}
 Then $\{e_1,e_2,e_3\}$ is a $k$-linear bases of the $k$-algebra
                        $\mathcal{E}$. The multiplication on $\mathcal{E}$ is defined by the following relations
                       $$\begin{cases} e_1\cdot e_i=e_i\cdot e_1=e_i, i=1,2,3 \\
                        e_2^2=e_3, e_2\cdot e_3=e_3\cdot e_2 =0, \\
                       e_3^2=0
                       \end{cases} .$$
Hence $\mathcal{E}$ is a local commutative $k$-algebra isomorphic to $k[x]/(x^3)$.
 Since $\{e_1,e_2,e_3\}$ is a $k$-linear basis of $\mathcal{E}$, any $k$-linear map $\sigma: \mathcal{E}\to \mathcal{E}$ uniquely corresponds to a matrix in $C_{\sigma}=(c_{ij})_{4\times 4}\in M_3(\k)$
with $$\left(
                         \begin{array}{c}
                          \sigma(e_1) \\
                          \sigma(e_2)\\
                          \sigma(e_3)\\
                         \end{array}
                       \right) =      C_{\sigma} \cdot \left(
                         \begin{array}{c}
                          e_1 \\
                          e_2\\
                          e_3\\
                         \end{array}
                       \right).  $$
                       Such $\sigma \in \mathrm{Aut}_k(\mathcal{E})$ if and only if
                      $$C_{\sigma}\in \mathrm{GL}_3(\k) \quad \text{and}\quad \sigma (e_i\cdot e_j)=\sigma (e_i)\sigma (e_j), \,\,\text{for any}\,\, i,j=1,2,3.$$
Therefore, $\sigma \in \mathrm{Aut}_k(\mathcal{E})$ if and only if
$$\begin{cases}
|(c_{ij})_{3\times 3}|\neq 0, \sigma(e_1)=e_1  \\
 [\sigma(e_2)]^2=\sigma(e_3), [\sigma(e_3)]^2=0\\
\sigma(e_2)\cdot \sigma(e_3)=\sigma(e_3)\cdot\sigma(e_2)=0
\end{cases} \Longleftrightarrow \begin{cases}
c_{22}\neq 0, c_{11}=1, c_{12}=c_{13}=0 \\
c_{21}=c_{31}=c_{32}=0\\
c_{33}=c_{22}^2.
\end{cases}
$$
Then we get \begin{align*}
\mathrm{Aut}_k(\mathcal{E})&\cong  \left\{ \left(
                         \begin{array}{ccc}
                           1 & 0& 0\\
                           0 & a & b \\
                           0 & 0 & a^2\\
                         \end{array}
                       \right)\quad | \quad a\in k^{\times}, b\in k \right\}\\
                       &\cong [k \rtimes \mathrm{SL}_1(k)]\rtimes k^{\times}.
\end{align*}
Since $\mathcal{E}$ is commutative, we have $\mathrm{Aut}_k(\mathcal{E})\cong \mathrm{Out}_k(\mathcal{E})$.
 By Remark \ref{imprem}, we have $\mathrm{Pic}_k (\mathcal{E})\cong \mathrm{Out}_k(\mathcal{E})$ and $\mathrm{DPic}_k(\mathcal{E})=\Bbb{Z}\times \mathrm{Pic}_k (\mathcal{E})$.
Thus
$$\mathrm{Pic}_k (E)\cong \mathrm{Pic}_k (\mathcal{E})\cong  [k \rtimes \mathrm{SL}_1(k)]\rtimes k^{\times},$$ and
 Theorem \ref{dpd} implies  $$\mathrm{DPic}(A)\cong \mathrm{DPic}_k (E)\cong \Bbb{Z}\times [k \rtimes \mathrm{SL}_1(k)]\rtimes k^{\times}.$$
\end{proof}

\begin{exa}\label{dualg}\cite{MHLX}
Let $A$ be the DG down-up algebra such that $$A^{\#}=\frac{k\langle x,y\rangle}{\left(\begin{array}{c}
   x^2y+(1-\xi)xyx-\xi yx^2  \\
   xy^2+(1-\xi)yxy-\xi y^2x  \\
\end{array}\right)},|x|=|y|=1,\xi^3=1, \xi\neq 1$$ and the differential $\partial_{A}$ is defined by
$\partial_{A}(x)=y^2$ and $\partial_{A}(y)=0.$  We have $$H(A)=\frac{\k\langle \lceil xy+yx\rceil, \lceil y\rceil\rangle}{\left(\begin{array}{ccc}
                                  \xi \lceil y\rceil \lceil xy+yx\rceil - \lceil xy+yx\rceil \lceil y\rceil \\
                                   \lceil y^2\rceil \\
                                    \end {array}\right)}$$
                                      by \cite[Proposition 5.5]{MHLX}.
                                     One sees that $H(A)$ is not Koszul and $\mathrm{gl.dim}H(A)=\infty$. In spite of this, $A$ is a Koszul Calabi-Yau DG algebra
by \cite[Proposition 6.1]{MHLX}. We can still apply Theorem \ref{dpd} to compute its derived Picard group.
\end{exa}
\begin{prop}
Let $A$ be the connected cochain DG algebra in Example \ref{dualg}.
Then $$\mathrm{DPic}(A)\cong \Bbb{Z}\times [k \rtimes \mathrm{SL}_1(k)]\rtimes k^{\times}.$$
\end{prop}
\begin{proof}
By \cite[Proposition 6.1]{MHLX}, $A$ is a Koszul Calabi-Yau DG algebra. From the proof of \cite[Proposition 6.1]{MHLX}, one sees that ${}_Ak$ admits a minimal semi-free resolution $F$ with $$F^{\#}=A^{\#}\oplus A^{\#}e_y\oplus A^{\#}e_z$$ and a differential $\partial_F$ defined by $\partial_F(\Sigma e_y)=y, \partial_F(e_t)=x+y\Sigma e_y$.
By the minimality of $F$, we have $$H(\Hom_A(F,k))=\Hom_A(F,k)= k\cdot 1^*\oplus k\cdot(\Sigma e_y)^*\oplus k \cdot(\Sigma e_z)^*.$$
  So the Ext-algebra $E=H(\Hom_A(F,F))$  is concentrated in degree $0$. On the other hand, $$\Hom_A(F,F)^{\#}\cong (k\cdot 1^*\oplus k\cdot(\Sigma e_y)^*\oplus k \cdot(\Sigma e_z)^*)\otimes_{k} F^{\#}$$ is concentrated in degree $\ge 0$. This implies that $E= Z^0(\Hom_A(F,F))$.
Since $F^{\#}$ is a free graded $A^{\#}$-module with a basis
$\{1, \Sigma e_y, \Sigma e_z\}$ concentrated in degree $0$,
  the elements in  $\Hom_A(F,F)^0$ is one to one correspondence with the matrixes in $M_3(k)$. Indeed, any $f\in \Hom_A(F,F)^0$ is uniquely determined by
  a matrix $A_f=(a_{ij})_{3\times 3}\in M_3(\k)$ with
$$\left(
                         \begin{array}{c}
                          f(1) \\
                          f(\Sigma e_y)\\
                          f(\Sigma e_z)\\
                         \end{array}
                       \right) =      A_f \cdot \left(
                         \begin{array}{c}
                          1 \\
                           \Sigma e_y\\
                           \Sigma e_z\\
                         \end{array}
                       \right).  $$
                       And $f\in  Z^0(\Hom_A(F,F)$ if and only if $\partial_{F}\circ f=f\circ \partial_{F}$, if and only if
 $$ A_f\cdot \left(
                         \begin{array}{ccc}
                           0 & 0& 0\\
                           y & 0 & 0\\
                           x & y & 0 \\
                         \end{array}
                       \right) =  \left(
                         \begin{array}{cccc}
                             0 & 0 & 0\\
                           y & 0 & 0\\
                           x & y & 0 \\
                         \end{array}
                       \right) \cdot A_f, $$  which is also equivalent to
                       $$\begin{cases}
                       a_{12}=a_{13}=a_{23}=0\\
                       a_{11}=a_{22}=a_{33}\\
                       a_{21}=a_{32}
                       \end{cases}$$
by direct computations. Hence the algebra $$ E\cong \left\{ \left(
                         \begin{array}{ccc}
                           a & 0& 0\\
                           b & a & 0 \\
                           c & b & a \\
                         \end{array}
                       \right)\quad | \quad a,b,c\in k \right\} = \mathcal{E}.$$
                       Set \begin{align*} e_1= \left(
                         \begin{array}{ccc}
                           1 & 0& 0\\
                           0 & 1 & 0 \\
                           0 & 0 & 1 \\
                         \end{array}
                       \right),& e_2= \left(
                         \begin{array}{ccc}
                           0 & 0& 0\\
                           1 & 0 & 0 \\
                           0 & 1 & 0 \\
                         \end{array}
                       \right),
                        e_3= \left(
                         \begin{array}{ccc}
                           0 & 0& 0\\
                           0 & 0 & 0 \\
                           1 & 0 & 0 \\
                         \end{array}
                       \right).
                       \end{align*}
 Then $\{e_1,e_2,e_3\}$ is a $k$-linear bases of the $k$-algebra
                        $\mathcal{E}$. The multiplication on $\mathcal{E}$ is defined by the following relations
                       $$\begin{cases} e_1\cdot e_i=e_i\cdot e_1=e_i, i=1,2,3 \\
                        e_2^2=e_3, e_2\cdot e_3=e_3\cdot e_2 =0, \\
                       e_3^2=0
                       \end{cases} .$$
Hence $\mathcal{E}$ is a local commutative $k$-algebra isomorphic to $k[x]/(x^3)$.
 Since $\{e_1,e_2,e_3\}$ is a $k$-linear basis of $\mathcal{E}$, any $k$-linear map $\sigma: \mathcal{E}\to \mathcal{E}$ uniquely corresponds to a matrix in $C_{\sigma}=(c_{ij})_{4\times 4}\in M_3(\k)$
with $$\left(
                         \begin{array}{c}
                          \sigma(e_1) \\
                          \sigma(e_2)\\
                          \sigma(e_3)\\
                         \end{array}
                       \right) =      C_{\sigma} \cdot \left(
                         \begin{array}{c}
                          e_1 \\
                          e_2\\
                          e_3\\
                         \end{array}
                       \right).  $$
                       Such $\sigma \in \mathrm{Aut}_k(\mathcal{E})$ if and only if
                      $$C_{\sigma}\in \mathrm{GL}_3(\k) \quad \text{and}\quad \sigma (e_i\cdot e_j)=\sigma (e_i)\sigma (e_j), \,\,\text{for any}\,\, i,j=1,2,3.$$
Therefore, $\sigma \in \mathrm{Aut}_k(\mathcal{E})$ if and only if
$$\begin{cases}
|(c_{ij})_{3\times 3}|\neq 0, \sigma(e_1)=e_1  \\
 [\sigma(e_2)]^2=\sigma(e_3), [\sigma(e_3)]^2=0\\
\sigma(e_2)\cdot \sigma(e_3)=\sigma(e_3)\cdot\sigma(e_2)=0
\end{cases} \Longleftrightarrow \begin{cases}
c_{22}\neq 0, c_{11}=1, c_{12}=c_{13}=0 \\
c_{21}=c_{31}=c_{32}=0\\
c_{33}=c_{22}^2.
\end{cases}
$$
Then we get \begin{align*}\mathrm{Aut}_k(\mathcal{E}) &\cong \left \{ \left(
                         \begin{array}{ccc}
                           1 & 0& 0\\
                           0 & a & b \\
                           0 & 0 & a^2\\
                         \end{array}
                       \right)\quad | \quad a\in k^{\times}, b\in \k \right\}\\
                       &\cong [k \rtimes \mathrm{SL}_1(k)]\rtimes k^{\times}.
\end{align*}
Since $\mathcal{E}$ is commutative, we have $\mathrm{Aut}_k(\mathcal{E})\cong \mathrm{Out}_k(\mathcal{E})$.
 By Remark \ref{imprem}, we have $\mathrm{Pic}_k (\mathcal{E})\cong \mathrm{Out}_k(\mathcal{E})$ and $\mathrm{DPic}_k(\mathcal{E})=\Bbb{Z}\times \mathrm{Pic}_k (\mathcal{E})$.
Hence
$$\mathrm{Pic}_k (E)\cong [k \rtimes \mathrm{SL}_1(k)]\rtimes k^{\times}, $$ and
Theorem \ref{dpd} implies  $$\mathrm{DPic}(A)\cong \mathrm{DPic}_k (E)\cong \Bbb{Z}\times [k \rtimes \mathrm{SL}_1(k)]\rtimes k^{\times}.$$

\end{proof}

\section*{Acknowledgments} The first author is supported by grants from NSFC (No. 11871326), Innovation Program of
Shanghai Municipal Education Commission (No. 12YZ031), the Key Disciplines of Shanghai Municipality (No. S30104) and the Discipline project at the corresponding level of Shanghai (No. A.13010112005).
The third author is supported by NSFC (No. 11571239) and a grant from NSF of Zhejiang province (No. LY14A01006).


\bibliography{}

\end{document}